\documentclass[11pt]{article}

\usepackage[top=1in, bottom=1in, left=1.25in, right=1.25in, marginparwidth=1in, marginparsep=0.1in]{geometry}

\usepackage{amsmath,amssymb,amsthm,xcolor,enumerate}
\usepackage[unicode,breaklinks=true,colorlinks=true,citecolor = {magenta}]{hyperref}

\numberwithin{equation}{section}
\newtheorem{theorem}{Theorem}[section]

\newtheorem{lemma}[theorem]{Lemma}
\newtheorem{proposition}[theorem]{Proposition}

\newtheorem{remark}[theorem]{Remark}

\newcommand{\bke}[1]{\left ( #1 \right )}

\newcommand{\norm}[1]{ \| #1  \|}

\newcommand\al{\alpha}
\newcommand\be{\beta}
\newcommand\ga{\gamma}
\newcommand\de{\delta}
\newcommand\ep{\epsilon}

\renewcommand\th{\theta}

\newcommand\la{\lambda}

\newcommand\ph{\varphi} 

\newcommand\om{\omega}
\newcommand\Ga{\Gamma}
\newcommand\De{\Delta}

\newcommand\La{\Lambda}

\newcommand{\Om}{\Omega}
\newcommand{\R}{\mathbb{R}}

\newcommand{\N}{\mathbb{N}}

\renewcommand{\div}{\mathop{\rm div}}

\newcommand{\supp} {\mathop{\mathrm{supp}}}

\newcommand{\dist} {\mathop{\mathrm{dist}}}

\newcommand{\td}{\tilde}

\renewcommand{\bar}[1]{\overline{#1}}
\newcommand{\lec}{{\ \lesssim \ }}

\newcommand{\cR}{\mathcal{R}}

\newcommand{\Leb}{{\operatorname{Leb}}}

\newcommand{\EQ}[1]{\begin{equation}\begin{split} #1 \end{split}\end{equation}}
\newcommand{\EQN}[1]{\begin{equation*}\begin{split} #1 \end{split}\end{equation*}}
\newcommand{\pa}{\partial}
\newcommand{\na}{\nabla}

\newcommand{\Del}{\Delta}
\newcommand{\del}{\delta}
\newcommand{\ddt}{\frac{d}{dt}}

\begin{document}

\title{Strong ill-posedness of logarithmically regularized\\ 2D Euler equations in the borderline Sobolev Space}
\author{Hyunju Kwon\footnote{
Hyunju Kwon, 
Department of Mathematics, University of British Columbia, 
Vancouver, BC V6T 1Z2, Canada;
e-mail: hkwon@math.ubc.ca}
}
\date{}

\maketitle

\begin{abstract}
Logarithmically regularized 2D Euler equations are active scalar equations with the non-local velocity $u = \na^\perp \De^{-1}T_\ga \om$ for the scalar $\om$. Two types of the regularizing operator $T_\ga$ with a parameter  $\ga> 0$ are considered: $T_\ga = \ln^{-\ga} (e+|\na|)$ and $T_\ga = \ln^{-\ga} (e-\De)$. These models regularize the 2D Euler equation for the vorticity (conventionally corresponding to the $\ga=0$ case), which results in their local well-posedness in the borderline Sobolev space $H^1(\R^2)\cap\dot{H}^{-1}(\R^2)$ when $\ga>\frac 12$. In this paper, we examine the regularized models in the remaining regime $\ga\le \frac 12$  and establish the strong ill-posedness in the borderline space. This completely solves the well-posedness problem of the regularized models in the borderline space by closing the gap between the local well-posedness result for $\ga>\frac 12$ and the strong ill-posedness for $\ga = 0$. 
\end{abstract}

\section{Introduction}

The incompressible Euler equation describes the behavior of homogeneous, inviscid and volume-preserving fluids,
\begin{align*}
\begin{cases}
\pa_t u + (u\cdot \na )u +\na p =0 &\quad(x,t)\in\R^n \times \R\\
\div u =0\\
u|_{t=0} = u_0.
\end{cases}
\end{align*}
Two unknowns $u$ and $p$ present the fluid velocity and pressure, respectively. For simplicity, we often work in the vorticity formulation for the Euler. In particular, the vorticity $\om = -\pa_2 u_1 + \pa_1 u_2$ in the two-dimensional space solves \begin{align}\label{Euler.eq}
\begin{cases}
\pa_t \om + (u\cdot \na )\om =0 &\quad(x,t)\in\R^2 \times \R\\
u = \na^\perp \psi, \quad \De \psi = \om\\
\om|_{t=0} = \om_0, \tag{E}
\end{cases}
\end{align}
where $\na^\perp = (-\pa_2,\pa_1)$. Then, the velocity $u$ can be recovered from the  vorticity $\om$ by the Biot-Savart law,
\[
u(x,t) = \text{p.v. }\frac 1{2\pi} \int \frac{(x-y)^\perp}{|x-y|^2} \om(y,t) dy,
\]
where $x^\perp = (x_1,x_2)^\perp = (-x_2,x_1)$.

In the past decades, the local well-posedness of the Euler equations has been well established for solutions with suitable regularity. For example, based on the standard energy method, 
the local well-posedness holds in the Sobolev spaces $W^{s,p}(\R^n)$, $s>\frac np+1$, $s\geq 1$, \cite{Majda, CF88}. In the solution spaces with threshold regularity, however, the well-posedness of the Euler equation has been a long-standing open problem. To tackle this, many efforts have been made. One way of obtaining the well-posedness is to work on a relatively ``regular" solution space among the borderline spaces, in the sense that the velocity in such solution spaces is under control in the Lipschitz space. Then, the local well-posedness follows from the usual energy method. Indeed, the Euler equation in $\R^n$, $n\geq 2$, is known to be well-posed local-in-time in the critical Besov spaces $B^{\frac np+1}_{p,1}(\R^n)$ for $1<p\le \infty$, see \cite{Vishik98, Vishik99, Chae04, PP04}. However, the borderline Sobolev space $H^{\frac n2+1}(\R^n)$ is not included in these critical Besov spaces. In fact, the Lipschitz norm of the velocity in the critical Sobolev space is out of control because the Sobolev embedding barely fails. 

To get better understanding of the behavior of the Euler flows in the critical Sobolev space, regularized Euler equations are introduced, see \cite{CCW11, CW12}. 
In \cite{CW12}, Chae and Wu study the logarithmically regularized 2D Euler equations,
\begin{align}\label{main.eq}\tag{LE}
\begin{cases}
\pa_t \om + (u\cdot \na) \om =0,		&(x,t)\in \R^2\times\R \\
u = \na^{\perp}\psi, \quad \Del \psi = T\om, \\
\om|_{t=0} = \om_0,
\end{cases} 
\end{align}
with the Fourier multiplier $T(|\na|)$ satisfying
\EQ{\label{int.con}
\int_1^\infty \frac{T^2(r)}{r} dr <+\infty.
}
Such operator $T$ regularizes the velocity in the Euler vorticity equation \eqref{Euler.eq} at the level of logarithm of the Laplacian. The particular integrability assumption \eqref{int.con} on $T$ is imposed to guarantee the local well-posedness of the regularized model in the critical Sobolev space. 
As typical examples of $T$ satisfying \eqref{int.con}, we have 
\EQ{\label{T.operator}
\widehat{T_{\ga}\om}(k) = \ln^{-\ga}(e+|k|^2)\hat{\om}(k),
\quad
\widehat{T_{\ga}\om}(k) = \ln^{-\ga}(e+|k|) \widehat{\om}(k), \quad\forall k\in \R^2.
}
for $\ga>\frac 12$.
In this paper, we restrict our attention to these two typical cases in the extended region of $\ga$, $\ga>0$. From now on, we use the abbreviation \eqref{main.eq} only when $T=T_\ga$. Conventionally, the multiplier $T_\ga$ with $\ga=0$ is considered as the identity operator. In other words, \eqref{main.eq} with $\ga =0$ corresponds to the 2D Euler vorticity equation.

The global well-posedness result of the 2D Euler vorticity equation (the case $\ga=0$ in \eqref{main.eq}) in the subcritical spaces $W^{s,p}(\R^2)\cap \dot{H}^{-1}(\R^2)$, $p>\frac 2s$ can be extended to that of \eqref{main.eq} for $\ga\geq 0$ (See \cite{CW12}). It follows from the usual energy method which requires two key estimates: commutator estimates and Sobolev inequalities. The critical\ space is determined by the Sobolev embedding
\[
\norm{\na u}_\infty = \norm{D\na^\perp\De^{-1}T_\ga \om }_\infty \lec \norm{\om}_{W^{s,p}(\R^2)}, \quad p>\frac 2s.
\]
In \cite{CW12}, the regularized velocity $u=\na^\perp \De^{-1}T_\ga \om$ leads to the local well-posedness of \eqref{main.eq} even in the critical space $H^1(\R^2)\cap \dot{H}^{-1}(\R^2)$ for $\ga>\frac 12$. Then, for $\ga\geq \frac 32$, the global lifespan of the local-in-time solutions is obtained by Dong and Li \cite{Dong2015}. On the other hand, the \textit{strong ill-posedness} of 2D Euler equation ($\ga=0$) in the borderline space $H^1(\R^2)\cap \dot{H}^{-1}(\R^2)$ is established by Bourgain and Li \cite{BL15}.  Later, Elgindi and Jeong \cite{Elgindi2017} prove the ill-posedness for some special initial data on the torus $\mathbb{T}^2$ with a different approach based on Kiselev-\v{S}ver{\'a}k \cite{KS14}). However, the well-posedness of the regularized model \eqref{main.eq} in the intermediate regime $0<\ga \le \frac 12$ still remains open. 

In this paper, we prove that the logarithmically regularized 2D Euler equations \eqref{main.eq} for $0< \ga\le \frac 12$ are strongly ill-posed in the critical Sobolev space $H^1(\R^2)\cap \dot{H}^{-1}(\R^2)$. The ill-posedness in the strong sense is defined as in \cite{BL15}. Namely, for any given compactly supported smooth initial data, an arbitrarily small perturbation in the borderline space can be always found such that the perturbed solution leaves the borderline space instantaneously. Our result closes the gap between $\ga=0$ (ill-posed) and $\ga>\frac 12$ (well-posed) and give complete answers to well/ill-posedness questions of logarithmically regularized 2D Euler equations. Furthermore, it says that even for the regularized 2D Euler equation, the strong ill-posedness holds in the same critical space of the Euler. 

We consider two types of perturbations: one has the non-compact support and the other is compactly supported.  
\begin{theorem}[Non-compact case]\label{thm.noncpt}
	Let $0<\ga\leq \frac 12$ and $a\in C_c^{\infty}(\R^2)$. Then, for any $\ep>0$, we can find a small perturbation $\zeta\in C^{\infty}(\R^2)$ in the sense of
	\[
	\norm{\zeta}_{\dot{H}^1(\R^2)}+\norm{\zeta}_{L^1(\R^2)}+ \norm{\zeta}_{L^{\infty}(\R^2)} <\ep
	\]
	such that for the perturbed initial data from $a$, we have a unique classical solution $\om$ to \eqref{main.eq}
		\[
		\begin{cases}
		\pa_t \om + u\cdot \na \om =0,		&(x,t)\in \R^2	\times (0,1]\\
		u = \na^{\perp}\psi, \quad \Del \psi = T_{\ga}\om, \\
		\om|_{t=0} = a+\zeta,
		\end{cases}
		\] 
	satisfying $\om(\cdot,t)\in C^{\infty}(\R^2)$ for $0\leq t\leq 1$ and $\om\in C([0,1];L^1(\R^2)\cap L^{\infty}(\R^2))$, but the solution $\om$ leaves the critical Sobolev space instantaneously. i.e., for each $0< T\leq 1$,
\EQ{\label{infty.critical}
		\norm{\om}_{L^{\infty}([0,T];\dot{H}^1(\R^2))} = +\infty.
}	
\end{theorem}

\begin{remark} 
The strong ill-posedness requires that the perturbed solution doesn't exist in the critical space at any positive time. On the other hand, to identify the perturbed solution, we need its unique existence in some space. \end{remark}

\begin{remark} \label{infty.critical.d} The perturbed solution achieves \eqref{infty.critical} in the sense that there exists a sequence of disjoint sets $Q_n  = [t_s^n,t_e^n]\times O_n$ on which $\norm{\om_n}_{L^\infty_t\dot{H}^1(Q_n)} >n $, where $t_s^n<t_e^n$, $\lim_{n \to \infty} t_e^n=0$, and $O_n$ is an open bounded set in $\R^2$.  

\end{remark}

\begin{theorem}[Compact case]\label{thm.cpt}
	Let $0<\ga \leq \frac 12$ and $a\in C_c^{\infty}(\R^2)$ which is odd in $x_2$. Then, for any $\ep>0$, we can find a small perturbation $\zeta\in C_c(\R^2)$ in the sense of
		\[
		\norm{\zeta}_{\dot{H}^1(\R^2)}+ \norm{\zeta}_{L^{\infty}(\R^2)} 
		+\norm{\zeta}_{L^1(\R^2)} + \norm{\zeta}_{\dot{H}^{-1}(\R^2)}
		<\ep
		\]
	such that for the perturbed initial data from $a$, we have a unique solution $\om:\R^2\times [0,1]\to \R$ in $C([0,1];C_c(\R^2))$ to \eqref{main.eq} 
		\[
		\begin{cases}
		\pa_t \om + u\cdot \na \om =0,		&(x,t)\in \R^2	\times (0,1]\\
		u = \na^{\perp}\psi, \quad \Del \psi = T_{\ga}\om, \\
		\om|_{t=0} = a+\zeta,	
		\end{cases}
		\] 
	satisfying $L^{\infty}$-norm preservation, but	the solution $\om$ leaves the critical Sobolev space instantaneously.
	\end{theorem}

\begin{remark} The perturbed solution in Theorem \ref{thm.cpt} leaves the critical space in the sense of \eqref{infty.critical} and Remark \ref{infty.critical.d}. By its construction, it has a local regularity enough to be well-defined in $L^\infty_t \dot{H}^1(Q_n)$ for each $n\in \N$. 
\end{remark}

The proof follows the outline of the strong ill-posedness scheme for the 2D Euler equations, developed in \cite{BL15}. 
It consists of three steps: creation of large Lagrangian deformation, local inflation of the critical norm, and patching argument. The first two steps are for the local construction of the perturbation $\zeta$. We first construct a family of initial data whose corresponding deformation matrix $D\phi(\cdot,t)$ get larger in $L^\infty$ space at shorter time $t$. Then, we upgrade each initial data so that the corresponding solution has larger critical norm in shorter time. In the last step, we sequentially patch the initial data in the family in a way of minimizing the interaction between them. This makes the solution for the patched initial data, called the global solution, locally behaves like the local solutions and hence have the critical norm inflation property. 

Difficulties first arise in the local construction of the perturbation. The velocity $u=\na^{\perp}\De^{-1}T_{\ga}\om$ in \eqref{main.eq} is more regular than the one in the Euler but the critical space remains same. This makes it more difficult for local solutions to be inflated in the critical norm. Furthermore, one of the main ingredients of getting the larger Lagrangian deformation is missing--- an explicit forms of the kernels of $D\na^\perp \De^{-1}T_\ga$. To solve these issues, we find essentially sharp pointwise lower bounds of the kernel. What's more, we construct local initial data having increasingly higher frequencies. Along these lines, the desired local construction can be achieved. Then, the successful construction of \textit{non-compactly supported} perturbation follows as in \cite{BL15}, placing local solutions far from each other. However, for a \textit{compactly supported} perturbation, the genuine difficulty moves to the patching process of local solutions. The increasingly higher frequencies of local initial data are likely to intensify interaction between local solutions. Moreover, in order to have a compact support, the local solutions must be placed at an infinitesimal distance from each other eventually. This enhances the interaction further. In a worse case, the active interaction can make high frequencies of local solutions canceled out, so the norm inflation of local solutions can be destroyed after patching. On the other hand, increasingly higher frequencies of local solutions most likely help to create the norm inflation. In order to see what really happens, a sharp control of the propagation of the current local initial data is required under the presence of the previously chosen ones. This can be done based on a keen analysis of the non-local operators. As a result, it can be shown that the existing local solution does not destroy the norm inflation of the current local solution in a very short time. This approach is different from the one in \cite{BL15} based on the perturbation argument, and makes the behaviour of the solution more clear.

The outline of the paper is as follows. Based on the creation of large Lagrangian deformation (Section \ref{sec.lld}), local critical norm inflation (Section \ref{sec.norm.infl}), and patching argument (Section \ref{sec.patching}), we get the proof of Theorem \ref{thm.noncpt} in Section \ref{sec.non.comp}. Then, the compact case (Theorem \ref{thm.cpt}) follows in Section \ref{sec.comp}.

\section{Notations}\label{sec.not}
\begin{itemize}

\item 
For a point $x\in \R^2$ and a positive real number $R$, $B(x,R)$ is the Euclidean ball defined by
\[
B(x,R) = \{y\in \R^2: |x-y|<R\}.
\] 
For a set $A\subset \R^2$ and a positive real number $R$, a generalized ball $B(A,R)$ means 
\[
B(A,R) = \{y\in \R^2: |x-y|<R \text{ for some } x\in A\}. 
\] 
Obviously, when $A$ is a single point set, $A=\{x\}$, we have $B(A,R)=B(x,R)$.

\item For given two sets $A$ and $B$ in $\R^2  $, the distance between two sets is denoted by
\[
\dist(A,B) :=\inf\{|x-y|: x\in A \text{ and } y\in B\}. 
\]

\item For any function $f$ on $\R^2$, we denote the Fourier transform of $f$ by
\[
\hat{f}(k) = \int_{\R^2} f(x)e^{-ik\cdot x} dx, \quad k\in \R^2,
\] 
and its inverse Fourier transform by
\[
\check{f}(x) =\frac 1{(2\pi)^2} \int_{\R^2} \hat{f}(k)e^{ik\cdot x} dk.
\]
\item For any $1\leq p\leq \infty$, $\norm{\cdot}_{L^p(\R^2)}$ is the usual Lebesgue norm in $\R^2$ with its abbreviation $\norm{\cdot}_p$. For any $m\in \N$ and $1\leq p\leq \infty$, $\norm{\cdot}_{W^{m,p}(\R^2)}$ denotes the usual Sobolev norm in $\R^2$. In the case of $p=2$, we use $H^m(\R^2)= W^{m,2}(\R^2)$. The homogeneous Sobolev norm is defined by
\[
\norm{f}_{\dot{H}^s(\R^2)} = \left(\int_{\R^2}|k|^{2s}|\hat{f}(k)|^2 dk\right)^{\frac 12}, \quad\forall s\in \R,
\] 
which includes the definition of $\dot{H}^{-1}(\R^2)$-norm. We omit $(\R^2)$ in the expression of Sobolev norms, when the domain of a function is obvious.

\item  Given two comparable quantities $X$ and $Y$, the inequality $X\lesssim Y$ stands for $X\leq C Y$ for some positive constant $C$. In a similar way, $X\gtrsim Y$ denotes $X\geq C Y$ for some $C>0$. We write $ X \sim Y$ when both $X\lesssim Y$ and $Y\lesssim X$ hold. When the constants $C$ in the inequalities depend on some quantities $Z_1$, $\cdots$, $Z_n$, we use $\lesssim_{Z_1,\cdots,Z_n}$, $\gtrsim_{Z_1,\cdots,Z_n}$, and $\sim_{Z_1,\cdots,Z_n}$. On the other hand, we say $X \ll Y$ if $X\le \ep Y$ for some sufficiently small $\ep>0$. Similarly, $X\gg Y$ is defined.

\end{itemize}

Since we prove the strong ill-posedness of \eqref{main.eq} for each $0<\ga\le \frac 12$, we omit the dependence of $\ga$ below if it is not needed. Also, without mentioning, we assume $0<\ga\le \frac 12$.

\section{Large Lagrangian deformation}\label{sec.lld}

In this section, we find a family of initial data which has large Lagrangian deformation property. As we mentioned, one of the main ingredients is finding a sharp pointwise estimate of the kernel of the operator $-\pa_{12}\Del^{-1}T_{\ga}$ from below. We consider the case $T_{\ga}(|\na|)=\ln^{-\ga}(e-\Del)$ first.

\begin{lemma}\label{estimate.K}
	Let $\ga>0$ and $K_{12}$ be the kernel of the Fourier multiplier $-\pa_{12} \Del^{-1} \ln^{-\ga}(e-\Del)$. Then, for any $x = (x_1 ,x_2)\in \R^2$, $x_1 > 0$, $x_2 > 0$, we have
	\EQ{\label{estimate.K.ineq}
	K_{12}(x_1,x_2)\geq  \frac {C x_1x_2}{|x|^4}
	\ln^{-\ga}\left(e+\frac 1{|x|}\right)e^{-|x|^2}
}
	for some positive constant $C$ depending only on $\ga$.
\end{lemma}

\begin{proof}
	Using the equalities
	\EQN{
	\int_0^{\infty}e^{-|k|^2s}|k|^2 ds &=1, \quad\text{ for }k\neq 0,\\
	\frac 1{\Ga(\ga)}\int_0^{\infty} e^{-at} t^{\ga} \frac{dt}t &= a^{-\ga}, \quad\text{ for }a> 0, 
	}
	the Fourier transform of $K_{12}$ can be written as 
	\EQ{\label{pre.int.form}
	\widehat{K_{12}}(k)
	&=-\frac{k_1k_2}{|k|^2}\ln^{-\ga}(e+|k|^2)
	=\int_0^{\infty} e^{-|k|^2 s} (-k_1k_2)\ln^{-\ga}(e+|k|^2) ds\\
	&=\int_0^{\infty}\frac 1{\Ga(\ga)}\int_0^{\infty} (e+|k|^2)^{-t} e^{-|k|^2 s} (-k_1k_2) t^{\ga} \frac{dt}{t} ds\\
	&=\frac 1{\Ga(\ga)}
	\int_0^{\infty}\frac 1{\Ga(t)}
	\int_0^{\infty}e^{-e\be}\int_0^{\infty}(-k_1k_2)e^{-|k|^2(\be+s)} ds \be^t \frac{d\be}{\be}t^{\ga} \frac {dt}t, \qquad \forall k\neq 0.
	}
	
	Taking the inverse Fourier transform, the kernel $K_{12}(x)$, for any $x\neq 0$, can be expressed as an integral form:
	\EQN{
		K_{12}(x) 
		&=
		\frac 1{\Ga(\ga)} \int_0^{\infty} \frac 1{\Ga(t)}\int_0^{\infty} e^{-e\be}\left( \int_0^{\infty} \pa_{12} (e^{(s+\be)\Del }\del_0)(x) ds\right)
		\be^t \frac{d\be}{\be}
		t^{\ga}\frac{dt}t \\
		&\sim_{\ga} x_1 x_2
		\int_0^{\infty}  \frac 1{\Ga(t)}\int_0^{\infty} e^{-e\be}\left( \int_0^{\infty} \frac 1{(s+\be)^3} e^{-\frac{|x|^2}{4(s+\be)}}ds \right)
		\be^t \frac{d\be}{\be}
		t^{\ga}\frac{dt}t \\
		&=\frac{x_1 x_2}{|x|^4}
		\int_0^{\infty} \frac {|x|^{2t}}{\Ga(t)} \int_0^{\infty} e^{-e|x|^2\td\be}
		\left(\int_0^{\infty}\frac 1{(\td s+\td\be)^3} e^{-\frac{1}{4(\td s+\td\be)}}d\td s\right)
		\td\be^t \frac{d\td\be}{\td\be}
		t^{\ga}\frac{dt}t,
	}
where $e^{t\De}\del_0$ is the usual heat kernel. 
The last equality easily follows from the change of variables $\be = |x|^2 \td \be$ and $ s= |x|^2 \td s$. 
	
	Then, the integral in $\td s$ can be computed as
	\EQ{\label{simplify.K}
		\int_0^{\infty}\frac 1{(\td s+\td\be)^3} e^{-\frac{1}{4(\td s+\td\be)}}d\td s
		&=\int_{\td \be}^{\infty}\frac 1{\tau^3} e^{-\frac{1}{4\tau}}d\tau
		=\int_{\td \be}^{\infty}\frac 1{\tau} (4 e^{-\frac{1}{4\tau}})'d\tau\\
		&=\frac 4{\tau} e^{-\frac{1}{4\tau}}\bigg|^{\infty}_{\td\be}
		+\int_{\td \be}^{\infty}\frac 4{\tau^2} e^{-\frac{1}{4\tau}}d\tau\\
		&=16\left(1-e^{-\frac{1}{4\td\be}}-\frac 1{4\td\be} e^{-\frac{1}{4\td\be}}\right),
	}
	so that we simplify the integral form as  
	\[
	K_{12}(x)\sim_{\ga}
	\frac{x_1 x_2}{|x|^4}
	\int_0^{\infty} \frac {|x|^{2t}}{\Ga(t)} \int_0^{\infty} e^{-e|x|^2\td\be}
	\left(1-e^{-\frac{1}{4\td\be}}-\frac 1{4\td\be} e^{-\frac{1}{4\td\be}}\right)
	\td\be^t \frac{d\td\be}{\td\be}
	t^{\ga}\frac{dt}t,
	\quad\forall x\neq 0.
	\]
	 
	Now, for each $x=(x_1,x_2)$ with $x_1>0$ and $x_2>0$, we find the lower bound of the kernel. Indeed, the desired lower bound \eqref{estimate.K.ineq} follows from 
	\EQN{
		\int_0^{\infty} &\frac {|x|^{2t}}{\Ga(t)} \int_0^{\infty} e^{-e|x|^2\td\be}
		\left(1-e^{-\frac{1}{4\td\be}}-\frac 1{4\td\be} e^{-\frac{1}{4\td\be}}\right)
		\td\be^t \frac{d\td\be}{\td\be}
		t^{\ga}\frac{dt}t\\
		&\gtrsim e^{-|x|^2}\int_0^1 \frac{|x|^{2t}}{\Ga(t)}\int_0^{\frac 1{e}}
		\td\be^{t} \frac{d\td\be}{\td\be}
		t^{\ga}\frac{dt}t\\
		&\gtrsim e^{-|x|^2}
		\int_0^1 \frac{|x|^{2t}}{t\Ga(t)} t^{\ga}\frac{dt}t
		\gtrsim e^{-|x|^2}
		\int_0^1 |x|^{2t} 	t^{\ga}\frac{dt}t \\
		&\gtrsim_{\ga} \ln^{-\ga}\left(e+\frac 1{|x|}\right)e^{-|x|^2}.
	}

\end{proof}

Now, we consider the case of $T_{\ga}(|\na|)=\ln^{-\ga}(e+|\na|)$. To express the corresponding kernel as an integral form, we need the following identity.

\begin{lemma}\label{subo}(Subordination identity) For any $r\geq 0$, we have
	\[
	e^{-r} = \frac{1}{\sqrt{\pi}} \int_0^{\infty} e^{-\tau} e^{-\frac{r^2}{4\tau}} \tau^{-\frac 12} d\tau.
	\]
\end{lemma}

\begin{proof}
	By using Fourier transform, it is easy to see 
	\[
	e^{-r} = \frac 1{\pi}\int_{-\infty}^{\infty}\frac 1{1+\th^2}e^{i\th r} d\th, \quad\forall r\geq 0.
	\]
	Since we can write
	\[
	\frac 1{1+\th^2} = \int_0^{\infty}e^{-\tau}e^{-\tau \th^2 } d\tau,
	\]
	the result follows from interchanging the $d\th-d\tau$ integral. 
\end{proof}

\begin{lemma}\label{estimate.tdK}
 Let $\ga>0$ and $\td K_{12}$ be the kernel of the multiplier $-\pa_{12} \Del^{-1} \ln^{-\ga}(e+|\na|)$. Then, for any $x = (x_1 ,x_2)\in \R^2$, $x_1 > 0$, $x_2 > 0$, we have
	\EQ{\label{estimate.tdK.ineq}
	\td K_{12}(x_1,x_2)\geq  \frac {Cx_1x_2}{|x|^4}
	\ln^{-\ga}\left(e+\frac 1{|x|}\right)e^{-|x|^2}
	}
	for some positive constant $C$ depending only on $\ga>0$.
\end{lemma}

\begin{proof}
	
	As we did in Lemma \ref{estimate.K}, the Fourier transform of $\td K_{12}$ can be expressed as follows:
	\EQ{\label{pre.int.form2}
	\widehat{{\td K}_{12}}(k) 
	&=-\frac {k_1k_2}{|k|^2}\ln^{-\ga}(e+|k|)\\
	&= \frac 1{\Ga(\ga)} \int_0^{\infty}\frac 1{\Ga(t)} \int_0^{\infty} e^{-e\be} \int_0^{\infty}(-k_1k_2)e^{-k|\be|}e^{-|k|^2s} ds \be^t \frac{d\be}{\be} t^{\ga}\frac{dt}t, \quad\forall k\neq 0.
	}
	Using the identity in Lemma \ref{subo}, for $\be \geq 0$ we have
	\EQ{\label{suboo}
	e^{-|k|\be} = \frac 1{\sqrt{\pi}} \int_0^{\infty} e^{-\tau}e^{-\frac {|k|^2\be^2}{4\tau}}\tau^{-\frac 12}d\tau,
	}
	so that the kernel can be written as an integral form: for any $x\neq 0$,
	\EQN{
		\td K_{12}(x)
		&=
		\frac 1{\sqrt{\pi}\Ga(\ga)} \int_0^{\infty} \frac 1{\Ga(t)}\int_0^{\infty} e^{-e\be} 
		\int_0^{\infty} 
		\int_0^{\infty} e^{-\tau} (\pa_{12} e^{\left(\frac{\be^2}{4\tau}+s\right)\Del}\del_0)(x) \tau^{-\frac 12} d\tau ds
		\be^t \frac{d\be}{\be}
		t^{\ga}\frac{dt}t\\
		&\sim_{\ga} x_1x_2
		 \int_0^{\infty} \frac 1{\Ga(t)}
		 \int_0^{\infty} e^{-e\be}
		  \int_0^{\infty} e^{-\tau} 
		\int_0^{\infty}  
		 \frac 1{\left(\frac{\be^2}{4\tau}+s\right)^3}e^{-\frac{|x|^2}{4\left(\frac{\be^2}{4\tau}+s\right)}}
		  ds \tau^{-\frac 12}d\tau
		\be^t \frac{d\be}{\be}
		t^{\ga}\frac{dt}t\\
		&= \frac{x_1x_2}{|x|^4}
		\int_0^{\infty} \frac {|x|^t}{\Ga(t)}
		\int_0^{\infty} e^{-e|x|\td\be}
		\int_0^{\infty} e^{-\tau} 
		\int_0^{\infty}  
		\frac 1{\left(\frac{\td\be^2}{4\tau}+\td s\right)^3}
		e^{-\frac{1}{4\left(\frac{\td\be^2}{4\tau}+\td s\right)}} 
		 d\td s \tau^{-\frac 12}d\tau
		\td\be^t \frac{d\td\be}{\td\be}
		t^{\ga}\frac{dt}t.
	}
	In the last equality, we do the change of variables $\be = |x|\td\be$ and $ s= |x|^2 \td s$.
	
	The integral in $\td s$ can be simplified as
	\EQ{\label{simplify.tdK}
		\int_0^{\infty}  
		\frac 1{\left(\frac{\td\be^2}{4\tau}+\td s\right)^3}
		e^{-\frac{1}{4\left(\frac{\td\be^2}{4\tau}+\td s\right)}} 
		d\td s	
		&=16
		\left(1-e^{-\frac{\tau}{\td\be^2}}-\frac{\tau}{\td\be^2} e^{-\frac{\tau}{\td\be^2}}\right),
	}
	and the integral form also becomes simple, 
	\[
	\td K_{12}(x)
	\sim_{\ga}
	\frac{x_1x_2}{|x|^4}
	\int_0^{\infty} \frac {|x|^t}{\Ga(t)}
	\int_0^{\infty} e^{-e|x|\td\be}
	\left( \int_0^{\infty} e^{-\tau} 
	(1-e^{-\frac{\tau}{\td\be^2}}-\frac{\tau}{\td\be^2} e^{-\frac{\tau}{\td\be^2}}) \tau^{-\frac 12}d\tau\right)
	\td\be^t \frac{d\td\be}{\td\be}
	t^{\ga}\frac{dt}t.
	\]
	
	To get a lower bound, we first consider the integral in $\tau$ and $\td\be$:
	\EQN{
	\int_0^{\infty} e^{-e|x|\td\be}
	&\left( \int_0^{\infty} e^{-\tau} 
	(1-e^{-\frac{\tau}{\td\be^2}}-\frac{\tau}{\td\be^2} e^{-\frac{\tau}{\td\be^2}}) \tau^{-\frac 12}d\tau\right)
	\td\be^t \frac{d\td\be}{\td\be}\\
	&\gtrsim
	\int_0^{\infty}e^{-\tau}e^{-\sqrt{e\tau}|x|}
	\left(\int_0^{\sqrt{\frac{\tau}{e}}}  \td\be^t \frac{d\td\be}{\td\be}\right)
	\tau^{-\frac 12}d\tau\\
	&\geq \frac 1{t\sqrt{e}^t}e^{-\frac{|x|}{e}} \int_0^{\frac 1{e^3}}e^{-\tau}
	{\tau}^{\frac{t-1}2}d\tau 
	\gtrsim \frac 1{t(t+1)e^{2t}}e^{-\frac{|x|}{e}}, \quad\forall x\neq 0, t>0.
	}

	Then, for each $x=(x_1,x_2)\in\R^2$ with $x_1>0$ and $x_2>0$, the desired lower bound \eqref{estimate.tdK.ineq} of the kernel follows from 
	\EQN{
		\int_0^{\infty} \frac {|x|^{t}}{\Ga(t)} 
		\int_0^{\infty} &e^{-e|x|\td\be}
		\left( \int_0^{\infty} e^{-\tau} 
		(1-e^{-\frac{\tau}{\td\be^2}}-\frac{\tau}{\td\be^2} e^{-\frac{\tau}{\td\be^2}}) \tau^{-\frac 12}d\tau\right)
		\td\be^t \frac{d\td\be}{\td\be}
		t^{\ga}\frac{dt}t\\
		&\gtrsim e^{-\frac{|x|}{e}}\int_0^{1} \frac {|x|^{t}}{t\Ga(t)}
		\frac 1{(t+1)e^{2t}}
		t^{\ga}\frac{dt}t 
		\gtrsim_\gamma \ln^{-\ga}\left(e+\frac{1}{|x|}\right)e^{-|x|^2}.
	}
	
\end{proof}

\begin{remark} By Lemma \ref{estimate.K} and Lemma \ref{estimate.tdK}, we can see that the kernels of $-\pa_{12}\De^{-1}T_\ga$ for both $T_\ga = \ln^{-\ga}(e-\De)$ and $T_\ga = \ln^{-\ga}(e + |\na|)$ have the same lower bound. Therefore, we use the combined notations $T_\ga$ and its kernel $K$ for both cases from now on. 
\end{remark}
\medskip

Now, we are ready to estimate Lagrangian deformation.

\begin{proposition}\label{large.lagrangian}
Let $\ga>0$. Suppose that a function $g\in C_c^{\infty}(\R^2)$ satisfies the following conditions.
\begin{enumerate}[(i)]
	\item $g$ is odd in $x_1$ and $x_2$.
	\item $g(x_1,x_2)\geq 0$ on $\{x_1\geq 0, x_2\geq 0\}$.
	\item 
	\[
	G\equiv \int_{x_1>0, x_2>0} g(x)\frac{x_1 x_2}{|x|^{4}} \ln^{-\ga}\left(e+\frac 1{|x|}\right) e^{-|x|^4} dx >0.
	\]
\end{enumerate}
Let $\phi$ be the characteristic line defined by
\[
\begin{cases}
	\pa_t \phi(x,t) = \na^{\perp}\Del^{-1}T_{\ga}\om (\phi(x,t), t)\\
	\phi(x,0) = x, 
\end{cases}
\]
where $\om$ is a smooth solution to \eqref{main.eq} for the initial data $\om_0=g$. 
Then, the Lagrangian deformation $D\phi$ satisfies
	\EQ{\label{large.lagrangian1}
	\int_0^t e^{-\norm{D\phi(\cdot,\tau)}_{\infty}^4} d\tau
	\leq \frac 1{CG} \ln(1+ CGt), \quad \forall t\geq 0
}
for some positive constant $C=C(\ga)$.
	In particular, we have
	\EQ{\label{large.lagrangian2}
	\max_{0\leq \tau\leq t}\norm{D\phi(\cdot,\tau)}_{\infty}
	\geq \ln^{\frac 14}\left(\frac{CGt}{\ln(1+CGt)}\right), 
	\quad \forall t>0. 
	}
\end{proposition}

\begin{proof}
	Using the parity of $g$, it can be easily checked that $\om$ is odd in $x_1$ and $x_2$, and hence $\phi(x,t)=(\phi_1(x_1,x_2,t),\phi_2(x_1,x_2,t))$ satisfies
	\begin{align}
	\phi_1(0,x_2,t)&\equiv 0,\quad \phi_2(x_1,0,t)\equiv 0 \label{phi.par} \qquad \forall x_1\in \R, x_2\in \R ,\\
	&\phi(0,t)\equiv 0. \nonumber
	\end{align}
	Also, the Frechet derivative $[Du(0,t)]_{ij}=\pa_j u_i(0,t)$ of $u=\na^{\perp}\Del^{-1}T_{\ga}\om$  at $x=0$ takes the form 
	\[
	Du(0,t)= 
	\begin{pmatrix}
	\la(t) & 0\\
	0 	   & -\la(t)	
	\end{pmatrix},
	\]
	where $\la(t)=-\pa_{12}\Del^{-1}T_{\ga}w(0,t)$. Then, this implies
	\[
	(D\phi)(0,t) = 
	\begin{pmatrix}
		\exp\left(\int_0^t \la(\tau) d\tau\right) & 0\\
		0	& \exp\left(-\int_0^t \la(\tau) d\tau\right)
	\end{pmatrix}.
	\]
	
	On the other hand, by \eqref{phi.par} and the sign preservation property of $\phi_1$ and $\phi_2$, we obtain for any $x_1\geq 0$, $x_2\geq 0$, and $t\geq 0$,
	\EQ{\label{phi.x}
	\frac 1{\norm{D\phi(\cdot,t)}_{\infty}} \phi_1(x_1,x_2,t) 
	&\leq x_1
	\leq \phi_1(x_1,x_2,t)\norm{D\phi(\cdot,t)}_{\infty},\\
	\frac 1{\norm{D\phi(\cdot,t)}_{\infty}} \phi_2(x_1,x_2,t) 
	&\leq x_2
	\leq \phi_2(x_1,x_2,t)\norm{D\phi(\cdot,t)}_{\infty}.
	}
	Thus, for any $x_1>0$, $x_2>0$, and $t\geq 0$,
	\EQN{
	\frac{\phi_1\phi_2}{\phi_1^2+\phi_2^2}
	=\frac 1{\frac{\phi_1}{\phi_2}+\frac{\phi_2}{\phi_1}} 
	\geq \frac 1{\norm{D\phi}_{\infty}^2}\frac{x_1x_2}{|x|^2}.	
	}
	
	Recall that we denote the kernel of the operator $-\pa_{12}\De^{-1}T_\ga$ by $K$.
By Lemma \ref{estimate.K} and Lemma \ref{estimate.tdK}, for any $x=(x_1,x_2)$ with $x_1>0$ and $x_2>0$, and $t\geq 0$,
	\EQN{
	K(\phi(x,t))
	&\gtrsim_{\ga} \left(\frac{\phi_1\phi_2}{|\phi|^2}\right) \frac 1{|\phi|^2} \ln^{-\ga}\left(e+\frac 1{|\phi|}\right) e^{-|\phi|^2}\\
	&\gtrsim \frac 1{\norm{D\phi}_{\infty}^{4}}\frac{x_1x_2}{|x|^{4}} 
	\ln^{-\ga}\left(e+\frac{\norm{D\phi}_{\infty}}{|x|}\right) e^{-\norm{D\phi}_{\infty}^2|x|^2}\\
	&\gtrsim \frac 1{\norm{D\phi}_{\infty}^{4}}\frac{x_1x_2}{|x|^{4}} 
	\ln^{-\ga}\left(e+\frac{1}{|x|}\right)\left(1+
	\ln\left(1+\norm{D\phi}_{\infty}\right) \right)^{-\ga}
	e^{-\frac 14\norm{D\phi}_{\infty}^4} e^{-|x|^4}\\
	&\gtrsim_{\ga} e^{- \norm{D\phi(\cdot,t)}_{\infty}^4}\frac{x_1x_2}{|x|^{4}}  
	\ln^{-\ga}\left(e+\frac{1}{|x|}\right)
	e^{-|x|^4}.
	}

	Now, we estimate $\la(t)$ from below
	\EQN{
	\la(t) 
	&=\int_{\R^2} K(y)w(y,t) dy=4\int_{y_1>0, y_2>0} K(y)w(y,t) dy\\
	&= 4\int_{x_1>0, x_2>0} K(\phi(x,t)) g(x) dx\\
	&\gtrsim_{\ga} e^{-\norm{D\phi(\cdot,t)}_{\infty}^4}
	\int_{x_1>0, x_2>0} g(x) \frac{x_1x_2}{|x|^{4}}  
	\ln^{-\ga}\left(e+\frac{1}{|x|}\right)
	e^{-|x|^4}  dx\\
	&= e^{-\norm{D\phi(\cdot,t)}_{\infty}^4} G.
	}

	Then, since 
	\[
	\norm{D\phi(\cdot,t)}_{\infty} \geq |D\phi(0,t)| \geq \exp\left(\int_0^t\la(\tau) d\tau\right), \quad\forall t\geq 0
	\]
	where $|\cdot|$ is the usual matrix norm,
	we have a positive constant $C>0$ depending only on $\ga$ such that
	\[
	\norm{D\phi(\cdot,t)}_{\infty} \geq \exp\left(\frac 1{4}CG \int_0^t e^{-\norm{D\phi(\tau)}_{\infty}^4} d\tau \right), \quad\forall t\geq 0.
	\]
	
	This implies that
	\EQN{
	\ddt \exp\left(CG \int_0^t e^{-  \norm{D\phi(\tau)}_{\infty}^4} d\tau\right)
	&=\exp\left(CG \int_0^t e^{- \norm{D\phi(\tau)}_{\infty}^4} d\tau\right)
	CG e^{- \norm{D\phi(t)}_{\infty}^4}\\
	&\leq CG \norm{D\phi(\tau)}_{\infty}^4 e^{- \norm{D\phi(\tau)}_{\infty}^4}
	\leq CG.
	}
	
	Therefore, we obtain
	\[
	\exp\left(CG \int_0^t e^{- \norm{D\phi(\tau)}_{\infty}^4} d\tau\right)
	\leq 1+CGt.
	\]
	The inequalities \eqref{large.lagrangian1} and \eqref{large.lagrangian2} then follows easily. 
\end{proof}

\begin{remark}\label{rist.lld} By a slight modification of the proof, we can restrict the region where the large Lagrangian deformation occurs;
	\EQ{
		\max_{0\leq \tau\leq t}\norm{D\phi(\cdot,\tau)}_{L^\infty(B(0,R))}
		\geq \ln^{\frac 14}\left(\frac{CGt}{\ln(1+CGt)}\right), 
		\quad \forall 0< t\leq 1,
	}
if $R>0$ satisfies 
\[
\supp(g)\subset B(0,R) \quad\text{and}\quad
\phi^{-1}(B_g,t)\subset B(0,R)
\]
for all $0\leq t\leq 1$, 
where $B_g=B(0,R_g)$ is the smallest ball containing $\bigcup_{0\leq t\leq 1}\supp(\om(\cdot,t))$. Indeed, if $x$ is in  $\supp(g)$, then $\phi(x,t)\subset \supp(\om(\cdot,t))$ and $|\phi(x,t)|\leq R_g$ when $0\le t\leq 1$. This implies that for $0\leq t\leq 1$
\[
\norm{D(\phi^{-1})(\cdot,t)}_{L^\infty(B_g)}
=\norm{(D\phi)^{-1}(\phi^{-1}(\cdot,t),t)}_{L^\infty(B_g)} 
\le \norm{D\phi(\cdot,t)}_{L^\infty(B(0,R))}.
\]
In the inequality, we use $|\det(D\phi(\cdot,t))|=1$ for any $t\geq 0$. Then, a modification of \eqref{phi.x} holds; for $x=(x_1,x_2)\in \supp(g)$, $x_1\geq 0$, $x_2\geq 0$, and $0\leq t\leq 1$, we have
\EQN{
	\frac 1{\norm{D\phi(\cdot,t)}_{L^\infty(B(0,R))}} \phi_1(x_1,x_2,t) 
	&\leq x_1
	\leq \phi_1(x_1,x_2,t)\norm{D\phi(\cdot,t)}_{L^\infty(B(0,R))},\\
	\frac 1{\norm{D\phi(\cdot,t)}_{L^\infty(B(0,R))}} \phi_2(x_1,x_2,t) 
	&\leq x_2
	\leq \phi_2(x_1,x_2,t)\norm{D\phi(\cdot,t)}_{L^\infty(B(0,R))}.
}
The rest of the proof is almost identical. 
\end{remark}

\section{Local critical Sobolev norm inflation}\label{sec.norm.infl}
In this section, we show that the inflation of the critical Sobolev norm can be induced from the largeness of Lagrangian deformation. Then, based on this, we construct a family of local solutions whose critical norm gets larger in a shorter time, while the critical norm of initial data gets smaller. 

We first recall Lemma 4.1 in \cite{BL15}.
\begin{lemma}\label{perb.char}
	Suppose $u=u(x,t)$ and $v=v(x,t)$ are smooth vector fields on $\R^2\times \R$. Let $\phi:\R^2\times \R \to \R^2$ and $\td\phi:\R^2\times \R \to \R^2$ be the solutions to 
	\[
	\begin{cases}
	\pa_t \phi(x,t) = u(\phi(x,t),t)\\
	\phi(x,0) = x
	\end{cases}
	\]
	and
	\[
	\begin{cases}
	\pa_t \td\phi(x,t) = u(\td\phi(x,t),t)+v(\td\phi(x,t),t)\\
	\td\phi(x,0) = x.
	\end{cases}
	\] 
	Then, we have positive constants $C$ and $C_1$ satisfying
	\EQN{
	\max_{0\leq t\leq 1}&(\norm{(\td\phi-\phi)(\cdot,t)}_{\infty}
	+\norm{(D\td\phi-D\phi)(\cdot,t)}_{\infty})\\
	&\leq C \max_{0\leq t\leq 1}\norm{v(\cdot,t)}_{W^{1,\infty}}\cdot\exp\bke{C_1\max_{0\leq t\leq 1}\norm{Dv(\cdot,t)}_{\infty}},
	}
	where $C$ depends on  $\norm{D^2u(\cdot,t)}_{L^{\infty}([0,1]\times\R^2)}$ and $\norm{Du(\cdot,t)}_{L^{\infty}([0,1]\times\R^2)}$, and $C_1$ is an absolute constant. 
\end{lemma}

The following is the main proposition in this section. 
\begin{proposition}\label{H1.norm.inflation.prop} Suppose that $\om$ is a smooth solution to \eqref{main.eq} with the initial data $\om_0$ and its velocity $u=-\na^{\perp}\De^{-1}T_\ga\om$, $\ga>0$, and satisfies the following properties.
	\begin{enumerate}[(i)]
		\item $\norm{\om_0}_{\infty}+\norm{\om_0}_1+\norm{\om_0}_{\dot{H}^{-1}}<\infty.$
		\item There exists $R_0>0$ such that
		\[
		\supp(\om_0)\subset B(0,R_0)
		\]
		and the characteristic line $\phi$, i.e., the solution to
		\[
		\begin{cases}
		\pa_t\phi(x,t)=u(\phi(x,t),t) &\R^2\times(0,\infty)\\
		\phi(x,0) = x				  &\R^2,
		\end{cases}
		\]
		satisfies
		\EQ{\label{lld.M}
		\norm{(D\phi)(\cdot,t_0)}_{L^\infty(B(0,R_0))}>L
		}
		for some $0<t_0\leq 1$ and $L> 8^9\cdot 10^6$. 
	\end{enumerate}
	 
		Then, we can construct a new smooth solution $\td \om$ to \eqref{main.eq} for a new initial data $\td \om_0$ which satisfies the following conditions.
	\begin{enumerate}[(i)]
		\item The size of the new initial data is controlled by that of the original one,
		\begin{align}
		\norm{\td\om_0}_{\dot{H}^{-1}} &\leq 2\norm{\om_0}_{\dot{ H}^{-1}}\label{IC.negative.norm}\\
		\norm{\td\om_0}_1 \leq 2\norm{\om_0}_1&,\quad
		\norm{\td\om_0}_{\infty} \leq 2\norm{\om_0}_{\infty},\label{IC.lp.norm}\\
		\norm{\td\om_0}_{\dot{H}^1}&\leq \norm{\om_0}_{\dot{H}^1} + L^{-\frac 12}.\label{IC.critical.norm} 
		\end{align}
		
		\item The new initial data is compactly supported,
		\EQ{\label{cpt.supp.tdom0}
		\supp(\td\om_0)\subset B(0,R_0).
	}

		\item A large Lagrangian deformation at $t_0$ induces $\dot{H}^1$-norm inflation:
		\EQ{\label{H1.inflation}
		\norm{\td\om(\cdot,t_0)}_{\dot{H}^1(\R^2)} >L^{\frac 13}.
		}	
	\end{enumerate}
\end{proposition}

\textit{Proof of the Proposition. }

\noindent\texttt{Sketch of the idea.}
Let $\td\phi$ be the characteristic line corresponding to the new smooth solution $\td\om$. Then, it solves
\[
\begin{cases}
\pa_t\td\phi(x,t)=\td u(\td \phi(x,t),t) &\R^2\times(0,\infty)\\
\td\phi(x,0) = x				  &\R^2,
\end{cases}
\]
where $\td u=\na^{\perp}\Del^{-1}T_{\ga}\td\om$. Since $\td \om(\td\phi(x,t),t)=\td\om_0(x)$, we can write $\dot{H}^1$-norm of $\td\om$ as
\EQ{\label{expressim.H1}
\norm{\na\td\om(\cdot,t)}_{2}^2 
=\int_{\R^2} |\na\td\om_0(x)\cdot(\na^{\perp}\td\phi_2)(x,t)|^2 dx
+\int_{\R^2} |\na\td\om_0(x)\cdot(\na^{\perp}\td\phi_1)(x,t)|^2 dx.
}

By Lemma \ref{perb.char}, if we choose a new initial data $\td \om_0$ to make $\norm{u-\td u}_{W^{1,\infty}}$ small, $\norm{D\phi-D\td\phi}_{\infty}$ also gets small. It follows that the main part in the right hand side of \eqref{expressim.H1} is the one in which $\td\phi$ is replaced by $\phi$.  
Then, we can produce the $\dot{H}^1$-norm inflation of $\td \om$ at $t_0$ from the largeness of Lagrangian deformation $D\phi$ in \eqref{lld.M} sense. Indeed, we construct the desired new initial data by adding a perturbation, localized at the point where the large Lagrangian deformation occurs, to the original initial data. 

\bigskip

\noindent\texttt{Step 1.} Construction of the new initial data $\td\om_0$.

Assume 
\[
\norm{\na \om(\cdot, t_0)}_{2}\leq L^{\frac 13}. 
\]
Otherwise, $\td\om_0 =\om_0$ completes the proof.

By the assumption \eqref{lld.M} and the smoothness of $\phi$, we can find $x_L=(x_L^1,x_L^2)$, $x_L^1x_L^2 \neq 0$, in $B(0,R_0)$ such that one of the entries of $D\phi(x_L,t_0)$, say $\pa_2\phi_2(x_L,t_0)$, satisfies
\[
|\pa_2\phi_2(x_L,t_0)|>L.  
\]
If we further use the continuity of $D\phi$, we can choose sufficiently small $\del>0$ satisfying $\del\ll \min(x_L^1, x_L^2)$, $B(x_L,\del)\subset B(0,R_0)$, and
\[
|\pa_2\phi_2(x,t_0)|>L, \qquad\forall |x-x_L|<\del. 
\]

Choose $\Psi$ be a smooth radial bump function which is compactly supported on the unit ball $B(0,1)$ and satisfies $\Psi\equiv 1$ on $B(0,\frac 12)$ and $0\leq \Psi\leq 1$. Set $\Psi_\del = \frac 1{\del}\Psi(\frac {x-x_L}{\del})$. By the choice of $x_L$ and $\del$, we note that the support of $\Psi_\del$ lies on one of the four quadrants. Now, let $b$ be the odd extension of $\Psi_\del$ in both variables. Then, we define the new initial data $\td\om_0$, adding a perturbation
\[
\eta_0(x) = \td\om_0(x) -\om_0(x) = \frac 1{20k\sqrt L}\cos (kx_1) b(x), 
\] 
to the original one $\om_0$ where $k$ will be chosen later sufficiently large. We can easily see that the perturbation $\eta_0$ is odd in both variables. 

\bigskip

\noindent\texttt{Step 2.} Check the required conditions on $\td \om$.

By its construction, the support of $\eta_0$ is contained in $B(0,R_0)$, so that \eqref{cpt.supp.tdom0} holds.    

To get \eqref{IC.negative.norm} and \eqref{IC.lp.norm}, we estimate the corresponding Sobolev norms of $\eta_0$,
\EQN{
\norm{\eta_0}_1 &\leq \frac 1{20k\sqrt L}\norm{b}_1\qquad
\norm{\eta_0}_{\infty} \leq \frac 1{20k\sqrt L}\norm{b}_{\infty}\\
&\norm{\eta_0}_{\dot{H}^{-1}}\lesssim
\norm{\widehat{x\eta_0}}_{\infty} + \norm{\eta_0}_2\lesssim \frac 1k,
}
where the estimate for the negative Sobolev norm follows from the parity of $\eta_0$. For sufficiently large $k$, both \eqref{IC.negative.norm} and \eqref{IC.lp.norm} hold true.

Finally, \eqref{IC.critical.norm} follows from
\[
\norm{b}_2 \leq 4\norm{\Psi_\del}_2 = 4\norm{\Psi}_2<4\sqrt{\pi},
\]
and
\[
\norm{\na\eta_0}_{2} 
\leq \frac 1{20k\sqrt{L}}\bke{k\norm{b}_2 + \norm{\na b}_2} \leq \frac 1{\sqrt L},
\]
provided that $k$ is sufficiently large.

Now, consider the $\dot{H}^1$-norm inflation of the new solution $\td\om$. As we mentioned, we first show that the perturbation in Lagrangian deformation is small. For this purpose, we consider the perturbation of velocity in $W^{1,\infty}(\R^2)$.

Since we have 
\EQ{\label{pert.vel}
\norm{\na (\td u-u)}_{\infty}
\lesssim_{\ga}(\norm{\na \td\om}_4+\norm{\na \om}_4)^{\frac 23}\norm{\td\om-\om}_2^{\frac 13},
}
it is enough to consider the terms on the right hand side. The terms $\norm{\na \td\om}_4$ and $\norm{\na \om}_4$ are estimated by the usual energy method. From the equation for $\td\om$, we have
\EQ{\label{eq.14}
\ddt \norm{\na \td\om}_4^4 \leq 4\norm{\na \td u}_{\infty}\norm{\na\td\om}_4^4.
}
By the log-type interpolation inequality,
\EQN{
\norm{\na \td u(\cdot,t)}_{\infty}
&\lesssim  1+ \norm{\td\om_0}_{\infty}\log (10+\norm{\td\om_0}_{2}+\norm{\na \td\om(\cdot,t)}_4^4),
}
we obtain 
\EQ{\label{bdd.14}
\max_{0\leq t\leq 1}\norm{\na \td\om(\cdot,t)}_{4}\leq C,
}
for some constant $C=C(\norm{\na \td\om_0}_4,\norm{\td\om_0}_{2})$. Note that we can choose an upper bound $C$ which is independent of $k$. Similarly, we have
\EQ{\label{bdd.14.td}
\max_{0\leq t\leq 1}\norm{\na \om(\cdot,t)}_{4}\leq C
}
for some positive constant $C$ independent of $k$. 

On the other hand, from the equations for $\td\om$ and $\om$, we get the equation for $\eta=\om-\td\om$, 
\[
\pa_t \eta + \na^{\perp}\Del^{-1}T_{\ga}\eta \cdot \na\om
+\na^{\perp}\Del^{-1}T_{\ga}\td\om \cdot \na \eta =0.
\]
Taking $\int \cdot \eta dx$ on both side, $\eta$ satisfies
\EQN{
\frac 12\ddt\norm{\eta(\cdot,t)}_2^2
&\leq \norm{\na^{\perp}\Del^{-1}T_{\ga}\eta}_{4} 
\norm{\na\om}_4\norm{\eta}_2
\lesssim \norm{\na\om}_4\norm{\eta}_2^2.
}
Here, the last inequality follows from Hardy-Littlewood Sobolev inequality and the compactness of the support of $\eta$. By Gr\"{o}nwall inequality, we obtain
\EQ{\label{est.eta.2}
\max_{0\leq t\leq 1}\norm{\eta(\cdot,t)}_2 
\lesssim \norm{\eta_0}_2
\lesssim \frac 1k. 
}
Combining with \eqref{pert.vel}, \eqref{bdd.14}, and \eqref{bdd.14.td}, the perturbation of $u$ can be estimated by
\[
\norm{\na(\td u- u)}_{\infty} \lesssim {k^{-\frac 13}}.
\]

Finally, by Gagliardo-Nirenberg interpolation inequality, for any $0\leq t\leq 1$, we have  
\[
\norm{(\td u - u)(\cdot,t)}_{\infty} \lesssim \norm{\na(\td u- u)}_{\infty}^{\frac 13}\norm{\td u - u}_4^{\frac 23}
\lesssim {k^{-\frac 19}} \norm{\eta}_{2}^{\frac 23}\lesssim k^{-\frac 79}.
\]
Therefore, Lemma \ref{perb.char} gives the desired estimate for the perturbation of Lagrangian deformation,
\[
\max_{0\leq t\leq 1}(\norm{(\td\phi-\phi)(\cdot,t)}_{\infty}
+\norm{(D\td\phi-D\phi)(\cdot,t)}_{\infty})
\lesssim k^{-\frac 13}.  
\]

Now, we are ready to get $\dot{H}^1$-norm inflation. 
Recall \eqref{expressim.H1} and we further estimate its right hand side as follows.
\EQ{\label{H1.norm.inflation.proof}
	\norm{\na\td\om(\cdot,t_0)}_{2}^2 
	\geq& \int_{\R^2} |\na\td\om_0(x)\cdot(\na^{\perp}\td\phi_2)(x,t_0)|^2 dx\\
	\geq& \ \frac 12\int_{\R^2} |\na\td\om_0(x)\cdot(\na^{\perp}\phi_2)(x,t_0)|^2 dx-O(k^{-\frac 23})\\
	\geq& \ \frac 14 \int_{\R^2} |\na\eta_0(x)\cdot(\na^{\perp}\phi_2)(x,t_0)|^2 dx\\
	&-\frac 12\int_{\R^2} |\na\om_0(x)\cdot(\na^{\perp}\phi_2)(x,t_0)|^2 dx-O(k^{-\frac 23}).
}

By the assumption on $\om$, we have
\[
\int_{\R^2} |\na\om_0(x)\cdot(\na^{\perp}\phi_2)(x,t_0)|^2 dx
\leq \norm{\na \om(\cdot,t_0)}_2^2 \leq L^{\frac 23}.
\]
On the other hand, by the construction of the perturbation $\eta_0$, we obtain
\EQN{
\int_{\R^2} |\na\eta_0(x)\cdot(\na^{\perp}\phi_2)(x,t_0)|^2 dx
&\geq \frac 1{800L}\int_{\R^2} |\sin(kx_1)b(x)\pa_2\phi_2(x,t_0)|^2 dx-O(k^{-2})\\
&\geq \frac {L}{800 }\frac 1{\del^2}\int_{|x-x_L|<\frac 12\del} |\sin(kx_1)|^2 dx-O(k^{-2})\\
&\geq \frac{1}{2^6\cdot 10^2}L -O(k^{-1}).
}

Therefore, we get the desired norm inflation
\[
\norm{\na \td\om(\cdot,t_0)}_2^2 \geq \frac {1}{2^8\cdot 10^2}L -\frac 12 L^{\frac 23} -O(k^{-\frac 23})
> L^{\frac 23}
\]
provided that $L>8^9\cdot 10^6$ and $k$ is sufficiently large. In other words, \eqref{H1.inflation} is obtained.   

\hfill$\square$

\begin{remark}\label{local.sol.family} Based on Proposition \ref{large.lagrangian} and Proposition \ref{H1.norm.inflation.prop}, we can construct a family of initial data having $\dot{H}^1$-norm inflation. 
	
	Choose a nonzero radial bump function $\varphi\in C_c^{\infty}(\R^2)$ satisfying $0\leq \varphi\leq 1$, $\varphi\equiv 1$ on $B(0,\frac 12)$, and $\supp(\varphi)\subset B(0,1)$. Then, we define $\rho\in C_c^{\infty}(\R^2)$ by
	\EQ{\label{defn.rho}
	\rho(x)=\rho(x_1,x_2)=\sum_{a_1,a_2=\pm 1} a_1a_2\varphi\left(\frac{x_1-a_1,x_2-a_2}{2^{-100}}\right).
	}
	Clearly, the function $\rho$ is odd in both variables, and  
	\[
	\int_{x_1>0, x_2>0} \rho(x)\frac{x_1 x_2}{|x|^{4}}e^{-|x|^4} dx>0.
	\]
	Now, for each $0<\ga\leq \frac 12$, define $g_A\in C_c^{\infty}(\R^2)$ by
	\EQ{\label{defn.gA}
	g_A(x)
	=
	\begin{cases}
	C_A\sum_{a_A\leq j < b_A}\frac 1{j^{\ga}}\rho(2^j x), &0<\ga<\frac 12\\[10pt]
	C_A\sum_{\ln A\leq j < A+\ln A} \frac 1{\sqrt j} \rho(2^j x), &\ga=\frac 12
	\end{cases}
	}
	where $C_A=\frac 1{\sqrt{\ln A}}\frac 1{\ln\ln A}$, $a_A=A^{\frac 1{1-2\ga}}$, and $b_A=(A+\ln A)^{\frac 1{1-2\ga}}$. Note that the summations in \eqref{defn.gA} are over integer $j$ in the range.
	
	First, $g_A$ satisfies all assumptions in Proposition \ref{large.lagrangian}. Obviously, $g_A$ is an odd function in $x_1$ and $x_2$, and $g_A(x_1,x_2)\geq 0$ for $x_1\geq 0$ and $x_2\geq 0$. Using disjoint supports of $\rho(2^j \cdot)$, $j\in \N$, we have for $A\geq e^2$, 
	\EQ{\label{estimate.B}
		G_A&=\int_{x_1>0, x_2>0} g_A(x)\frac{x_1 x_2}{|x|^{4}} \ln^{-\ga}\left(e+\frac 1{|x|}\right) e^{-|x|^4} dx\\
		&=
		C_A\sum_{j}\frac 1{j^{\ga}}
		\int_{x_1>0, x_2>0} \rho(2^j x)\frac{x_1 x_2}{|x|^{4}} \ln^{-\ga}\left(e+\frac 1{|x|}\right) e^{-|x|^4} dx\\
		&=
		C_A\sum_{j }\frac 1{j^{\ga}}
		\int_{\substack{x_1>0, x_2>0\\x\in \supp(\rho)}} \rho(x)\frac{x_1 x_2}{|x|^{4}} \ln^{-\ga}\left(e+\frac {2^j}{|x|}\right) e^{-\frac{|x|^4}{2^{4j}}}  dx\\ 
		&\geq
		C_A\sum_{j }\frac 1{j^{2\ga}} \left(
		\int_{x_1>0, x_2>0} \rho(x)\frac{x_1 x_2}{|x|^{4}} 
		e^{-|x|^4} dx\right)>0.
	}
	Here, the range of summation over $j$ depends on $\ga$, which follows to the one in \eqref{defn.gA}. 
	
	Since for $A\gg 1$, we have
	\EQN{
	\sum_{j}\frac 1{j^{2\ga}}
	&\sim 
	\begin{cases}
	\int_{a_A}^{b_A} \frac 1{x^{2\ga}} dx = \frac 1{1-2\ga} (b_A^{{1-2\ga}}-a_A^{1-2\ga})= \frac 1{1-2\ga} \ln A, &0<\ga<\frac 12\\[10pt]
	\int_{\ln A}^{A+\ln A} \frac 1{x} dx = \ln(A+\ln A) - \ln\ln A, &\ga =\frac 12
	\end{cases}\\
	&\sim_{\ga} \ln A, 
	} 
	$G_A$ has a lower bound 
	\EQN{
	G_A \gtrsim_{\ga} \frac{\sqrt{\ln A}}{\ln\ln A}.
 	}
	Then, by Proposition \ref{large.lagrangian}, for any $A$ with $A\geq A_0$ for some $A_0=A_0(\ga)$, we can find $t_A \in \left(0, \frac 1{\ln\ln A}\right]$  such that the characteristic line $\phi_A$ corresponding to each initial data $g_A$ has a large Lagrangian deformation
	\EQ{\label{lld.gA}
	\norm{D\phi_A(\cdot,t_A)}_{L^\infty(B(0,\frac12))}>\ln^{\frac 14}\ln\ln\ln A.
	}
	Now, we induce critical norm inflation from large Lagrangian deformation. Observe that all assumptions in Proposition \ref{H1.norm.inflation.prop} hold for $\om_0=g_A$, $t_0=t_A$, $L=\ln^{\frac 14}\ln\ln\ln A$, and $R_0=1$, provided that $A$ is sufficiently large. Indeed, using
	\[
	|\phi_A(x,t)-x|\leq \int_0^t |\pa_s\phi_A(x,s)| ds
	\leq \norm{\na^\perp\De^{-1}T_\ga (g_A\circ\phi_A^{-1})}_{L^\infty_{x,t}} t
	\lec \norm{g_A}_1^{\frac 12}\norm{g_A}_\infty^\frac 12 t
	\]
	for all $x\in \R^2$ and $ t\geq 0$,
we have $\phi^{-1}(B_{g_A},t)\subset B(0,1)$ for sufficiently large $A$, where $B_{g_A}$ is defined as in Remark \ref{rist.lld}.
In what follows, we have a desired family $\{\td g_A\}$ of a new initial data which has the following properties: 
	\begin{enumerate}[(i)]
		\item $\td g_A$ gets small as $A$ goes to infinity in the following sense:
		\EQ{\label{ini.tdgA}
			&\norm{\td g_A}_{1}\leq 2\norm{g_A}_1\lesssim  \frac{1}{A^{\ln 4}},\\
			&\norm{\td g_A}_{\infty}\leq 2\norm{g_A}_{\infty}\leq \frac{2}{\sqrt{\ln A}}\\
			\norm{\na \td g_A}_{2} \leq \norm{\na g_A}_{2} &+ \ln^{-\frac 18}\ln\ln\ln A
			\leq \frac{C_{\ga}}{\ln\ln A}+ \ln^{-\frac 18}\ln\ln\ln A
		}
		where $C_{\ga}$ is independent of $A$. 
		\item $\supp(\td g_A)\subset B(0,1)$.	
		\item The smooth solution $\td\om_A$ to \eqref{main.eq} for the initial data $\td g_A$ has local critical norm inflation:
		\[
		\norm{\na\td\om_A(\cdot,t_A)}_2 >\ln^{\frac 1{12}}\ln\ln\ln A.
		\]
	\end{enumerate}

\end{remark}

\section{Patching argument}\label{sec.patching}
In this section, we introduce useful lemmas and a proposition for the construction of the desired global solution from local ones. For the non-compactly supported case, our strategy is using a huge distance between local solutions so that they barely interact to each other. This leads the global solution to locally behave like local solutions. The following proposition describes this in detail. 

\begin{proposition}\label{prop.patching.noncompact}
	Let $\{\om_{j0}\}\subset C_c^{\infty}(B(0,1))$ be a sequence of functions satisfying
	\EQ{\label{initial.data.M}
	\sum_{j=1}^{\infty}(\norm{\om_{j0}}_{H^1}^2+\norm{\om_{j0}}_1) + \sup_j \norm{\om_{j0}}_{\infty}\leq M 
}
for some $M>1$. 
	For each $\ga>0$, let $C_0$ be an absolute constant such that
	\[
	\norm{\na^{\perp}\Del^{-1}T_{\ga}f}_{\infty}\leq C_0(\norm{f}_1+\norm{f}_{\infty}).
	\] 
	
	Then, we can find a sequence $\{x_j\}$ of centers with $|x_j-x_k|\gg 1$ for $j\neq k$ such that there exists a unique classical solution $\om$ to \eqref{main.eq} for the initial data  
	\[
	\om_0(x) = \sum_{j=1}^{\infty} \om_{j0}(x-x_j) \in L^1\cap L^{\infty}\cap H^1\cap C^{\infty}
	\]
	such that the following hold.
	\begin{enumerate}[(i)]
		\item For any $0\leq t\leq 1$, $\om(\cdot,t)$ is supported in the union of disjoint balls: 
		\EQ{\label{supp.om.dis}
		\supp(\om(\cdot,t))\subset \bigcup_{j=1}^{\infty}B(x_j, 3C_0M).
	}	
		\item For each $0\leq t\leq 1$, $\om(\cdot,t)\in C^{\infty}(\R^2)$, and $\om\in C([0,1];L^1(\R^2)\cap L^{\infty}(\R^2))$.
		\item For any $\ep>0$, we can find a sufficiently large integer $j_0=j_0(\ep)$ so that for $j\geq j_0$, we have
		\EQ{\label{local.behavior.prop}
		\max_{0\leq t\leq 1}\norm{(\om-\om_j)(\cdot,t)}_{H^2(B(x_j,3C_0M))} <\ep,
	}
		where a local solution $\om_j$ solves \eqref{main.eq} for the initial data
		\[
		\om_j|_{t=0} = \om_{j0}(\cdot-x_j). 
		\]
		\end{enumerate}
\end{proposition}

Before we prove this proposition, we consider some preliminary lemmas. 
\begin{lemma}\label{simple.patching}
	Suppose that $f\in H^k\cap L^1$ for some $k\geq 2$ and $g\in H^2\cap L^1$ satisfy
	\begin{gather}
	\norm{f}_1 + \norm{g}_1 + \sup(\norm{f}_{\infty},\norm{g}_{\infty}) \leq M <\infty, \nonumber\\
	\dist(\supp(f), \supp(g)) \geq 100 C_0M >0 \label{dist.fg}
	\end{gather}
for some constant $M>1$,	and the Lebesgue measure of the support of $f$ is bounded by some positive constant $M_1$.
	
	Then, the solution $\om$ to 
	\[
	\begin{cases}
	\pa_t\om + u\cdot \na\om =0 &\R^2 \times (0,1]\\
	u= \na^{\perp}\Del^{-1}T_{\ga}\om\\
	\om|_{t=0} = f+g
	\end{cases}
	\]
	has the following properties.
	\begin{enumerate}[(i)]
		\item The solution $\om$ can be decomposed as 
		$\om = \om_f+\om_g$ such that
		\begin{align}
		\om_f|_{t=0} = f&, \quad \om_g|_{t=0}=g \nonumber\\
		\supp(\om_f(\cdot,t))&\subset B(\supp(f), 2C_0M), \label{condition.decomp.f}\\
		\supp(\om_g(\cdot,t))&\subset B(\supp(g), 2C_0M), \label{condition.decomp.g}\\
		\dist(\supp(\om_f(\cdot,t)), &\supp(\om_g(\cdot,t))) \geq 90 C_0M, \quad\forall 0\leq t\leq 1, \label{condition.decomp.fg}
		\end{align}
		where $C_0$ is defined as in Proposition \ref{prop.patching.noncompact}.
		\item The Sobolev norms of $\om_f$ can be estimated by
		\EQ{\label{bdd.Hk}
		\max_{0\leq t\leq 1} \norm{\om_f(\cdot,t)}_{H^k} \leq C
	}
	for some constant $C=C(\norm{f}_{H^k},k,M, M_1)$ independent of $\norm{g}_{H^k}$. 
	\end{enumerate}
\end{lemma}

\begin{proof}
	Define $\om_f$ and $\om_g$ by the solutions to
	\EQ{\label{eqn.omf}
	\begin{cases}
	\pa_t \om_f + u\cdot \na \om_f =0 \\
	\om_f|_{t=0} = f
	\end{cases}
}
	and
	\EQ{\label{eqn.omg}
	\begin{cases}
	\pa_t \om_g + u\cdot \na \om_g =0 \\
	\om_g|_{t=0} = g.
	\end{cases}
	}
	Let $\phi$ be the characteristic line which solves
	\[
	\begin{cases}
	\pa_t\phi(x,t) = u(\phi(x,t),t)\\
	\phi(x,0) =x.
	\end{cases}
	\]
	Then, the equations \eqref{eqn.omf} and \eqref{eqn.omg} can be written as 
	\[
	\om_f(\phi(x,t),t)=f(x),\quad\text{and}\quad \om_g(\phi(x,t),t)=g(x).
	\]
	From these forms, it follows that for $1\leq p \leq \infty$ 
	\[
	\norm{\om_f(\cdot,t)}_p = \norm{f}_p, \quad\text{and}\quad
	\norm{\om_g(\cdot,t)}_p = \norm{g}_p, \quad
	\forall 0\leq t\leq 1,
	\]
 	and 
	\[
	\max_{0\leq t\leq 1}\norm{u(\cdot,t)}_{\infty} \leq C_0M.
	\]

Since we have
	\[
	|\phi(x,t)-x| \leq \int_0^t |\pa_s \phi(x,s)| ds \leq \max_{0\leq s\leq 1}\norm{u(\cdot,s)}_{\infty}t\leq C_0Mt,
	\]
\eqref{condition.decomp.f} and \eqref{condition.decomp.g} easily follows from
	\EQN{
	&\supp(\om_f(\cdot,t))\subset \phi(\supp(f), t)\subset B(\supp(f), 2C_0M), \\
	&\supp(\om_g(\cdot,t))\subset \phi(\supp(g), t)\subset B(\supp(g), 2C_0M),
	\quad \forall0\leq t\leq 1.
	}

		Using the assumption \eqref{dist.fg} additionally, the triangle inequality implies 
	\EQ{\label{dist.om.fg}
	\dist(\supp(\om_f(\cdot,t)), \supp(\om_g(\cdot,t))) \geq 90 C_0M, \quad\forall 0\leq t\leq 1. 
}
In other words, \eqref{condition.decomp.fg} is obtained. 
	
	To control the Sobolev norm of $\om_f$, we first estimate $\na^{\perp}\Del^{-1}T_{\ga}\om_g$ when $0\leq t\leq 1$ and $x\in \supp(\om_f(\cdot,t))$. 
	Since the supports of $\om_f(\cdot,t)$ and $\om_g(\cdot,t)$ are apart from each other for $0\leq t\leq 1$ (see \eqref{dist.om.fg}), we have for $0\leq t\leq 1$ and $x\in \supp(\om_f(\cdot,t))$, 
	\EQ{\label{remainder.velocity}
	\left|\pa^{\al}\na^{\perp}\Del^{-1}T_{\ga}\om_g(x,t)\right|
	&=\left|\int_{|y-x|\geq 90C_0M} \pa^{\al}H(x-y)\om_g(y) dy\right|\\
	&\leq \norm{\pa^{\al}H}_{L^{\infty}(|z|\geq 90C_0M)} \norm{g}_1,
	}
	where $H$ is the kernel of the Fourier multiplier $\na^{\perp}\Del^{-1}T_{\ga}$. By Lemma \ref{ker.H.lem}, for any multi-index $\al$ with $|\al|\ge 0$, $H$ satisfies 
	\[
	|\pa^{\al}H(z)| \lesssim_{\al,\ga} \frac1{|z|^{|\al|+1}}, \quad\forall z\neq 0
	\]  
	and therefore  
	\EQ{\label{est.ug.lemma1}
	\max_{0\leq t\leq 1}
	\max_{x\in\supp(\om_f(\cdot,t))}|\pa^{\al}\na^{\perp}\Del^{-1}T_{\ga}\om_g(x,t)|
	\lesssim_{\al,\ga} 1.
}
	
	To get \eqref{bdd.Hk}, we use the energy method. We consider the Sobolev norm $W^{1,p}(\R^2)$ for $2<p\leq +\infty$ first. From the equation \eqref{eqn.omf} for $\om_f$, we have 
	\EQ{\label{energy.est.p}
	\frac 1p\ddt \norm{\na\om_f}_p^p
	\leq (\norm{D\na^{\perp}\Del^{-1}T_{\ga}\om_f}_{\infty}
	+\norm{D\na^{\perp}\Del^{-1}T_{\ga}\om_g}_{L^\infty(\supp(\om_f(\cdot,t)))})\norm{\na\om_f}_p^p
}
	By log-type interpolation inequality together with $L^p$-norm preservation of $\om_f$, 
	\[
	\norm{D\na^{\perp}\Del^{-1}T_{\ga}\om_f(\cdot,t)}_{\infty}\lesssim_{p} 1+\norm{f}_{\infty}\log(10 + \norm{f}_2 + \norm{\na\om_f(\cdot,t)}_p^p), 
	\quad\forall 0\leq t\leq 1. 
	\]	
	Combining with \eqref{est.ug.lemma1} and \eqref{energy.est.p}, this implies
	\[
	\max_{0\leq t\leq 1}\norm{\om_f(\cdot,t)}_{W^{1,p}(\R^2)}\leq C(\norm{f}_{W^{1,p}(\R^2)}, p,M). 
	\]
	
	We now estimate in $H^k(\R^2)$, $k\geq 2$. By the commutator estimate in \cite[Theorem 1.9]{Li16}, for $J=(1-\Del)^{\frac 12}$, we get
	\EQN{
	\ddt \norm{J^k \om_f}_2
	\leq& \ \norm{[J^k,\na^{\perp}\Del^{-1}T_{\ga}\om_f\cdot \na]\om_f}_2
	+\norm{[J^k,\na^{\perp}\Del^{-1}T_{\ga}\om_g\cdot \na]\om_f}_2\\
	\lesssim& \ \norm{J^{k-1}D\na^{\perp}\Del^{-1}T_{\ga}\om_f}_{3}\norm{\na\om_f}_{6}
	+\norm{D\na^{\perp}\Del^{-1}T_{\ga}\om_f}_{\infty}\norm{J^k\om_f}_2\\
	&+\max_{|\al|\leq k}\max_{0\leq t\leq 1} \norm{D^{\al}\na^{\perp}\Del^{-1}T_{\ga}\om_g(\cdot,t)}_{L^{\infty}(\supp(\om_f(\cdot,t)))}\norm{J^k\om_f}_2	\\
	\leq& \ C\norm{J^k\om_f}_2, 
	}
where the constant in the last inequality depends on $\norm{f}_{H^2}$, $M$, $M_1$, and $k$. 

	Therefore, by Gr\"{o}nwall inequality, we obtain \eqref{bdd.Hk}.
		
\end{proof}
 
\begin{lemma} \label{simple.patching.2}
	Suppose that $f$ is in $H^3(\R^2)\cap L^1(\R^2)$ with $\Leb(\supp(f))\leq M_1$ for some $M_1$,  $g$ is in $H^2(\R^2)\cap L^1(\R^2)$, and they satisfy
	\EQN{
	\norm{f}_1 + \norm{g}_1 + \sup(\norm{f}_{\infty},\norm{g}_{\infty})\leq M
	}
	for some $M>1$.
	Let $\om$ and $\td\om$ be solutions to \eqref{main.eq} for the initial data $f+g$ and $f$, respectively.
	
	Then, for each $\ep>0$, we can find sufficiently large $R=R(\ep, \norm{f}_{H^3}, M,M_1)>0$ such that if 
	\EQ{\label{dist.ass}
	\dist(\supp(f),\supp(g))\geq R,
}
	then $\om$ can be decomposed as $\om = \om_f + \om_g$ such that $\om_f$ and $\om_g$ satisfy \eqref{condition.decomp.f}-\eqref{condition.decomp.fg} and
	\EQ{\label{local.behavior}
		\max_{0\leq t\leq 1}\norm{(\om_f-\td\om)(\cdot,t)}_{H^2}<\ep.
	}	
		
	\end{lemma}
\begin{remark} 	Similar to \eqref{condition.decomp.f} and \eqref{condition.decomp.g}, we have 
		\EQ{\label{supp.tdom}
		\supp(\td\om(\cdot,t)) \subset B(\supp(f), 2C_0M), 
		\quad\forall 0\le t\le 1,
		}
	where $C_0$ is defined as in Proposition \ref{prop.patching.noncompact}. It follows from $\max_{0\leq t\leq 1}\norm{\td u(\cdot,t)}_{\infty}\leq C_0M$ for $\td u = \na^{\perp}\Del^{-1}T_{\ga}\td\om$. 
\end{remark}		
	
	\begin{proof}
	We use the same decomposition $\om=\om_f+\om_g$ in Lemma \ref{simple.patching}. Then, we have \eqref{condition.decomp.f} and \eqref{condition.decomp.g}. Furthermore, \eqref{condition.decomp.fg} is also obtained, provided that $R\geq 100C_0M$. In fact, using  \eqref{dist.ass}, we have 
	\EQ{\label{dist.supp.omfg}
	\dist(\supp(\om_f(\cdot,t),\supp(\om_g(\cdot,t))))\geq R-10C_0 M\geq \frac 12 R, \quad \forall 0\leq t\le 1
}
for sufficiently large $R$.

	To get \eqref{local.behavior}, we recall the equation for $\om_f$, 
	\[
	\begin{cases}
	\pa_t \om_f + u \cdot \na \om_f =0\\
	\om_f |_{t=0} = f.
	\end{cases}
	\]	

By Gagliardo-Nirenberg inequality, 
	\EQ{\label{est.H2}
		\norm{(\om_f-\td\om)(\cdot,t)}_{H^2}
		&\lesssim (\norm{\om_f(\cdot,t)}_{H^3}+\norm{\td\om(\cdot,t)}_{H^3})^{\frac 23} \norm{(\om_f-\td\om)(\cdot,t)}_{2}^{\frac 13}.
	}
By Lemma \ref{simple.patching}, we obtained
	\EQ{\label{bdd.H3.omf}
	\max_{0\leq t\leq 1}\norm{\om_f(\cdot,t)}_{H^3}\leq C(\norm{f}_{H^3}, M, M_1).
	}
Also, by the usual energy method, we also have a similar inequality for $\td\om$
	\EQ{\label{bdd.H3.tdom}
	\max_{0\leq t\leq 1}\norm{\td\om(\cdot,t)}_{H^3}\leq C(\norm{f}_{H^3}, M, M_1).
}
	Therefore, it is enough to consider $\norm{\eta(\cdot,t)}_2$ for $\eta=\om_f-\td\om$. 
	
	The equation for $\eta$ is 
	\[
	\begin{cases}
	\pa_t\eta + \na^{\perp}\Del^{-1}T_{\ga}\td\om\cdot\na\eta + \na^{\perp}\Del^{-1}T_{\ga}\eta \cdot\na \om_f + \na^{\perp}\Del^{-1}T_{\ga}\om_g \cdot \na\om_f=0\\
	\eta|_{t=0}=0.
	\end{cases}
	\]
	Taking $\int\cdot\eta dx$ on both side of the first equation and using \eqref{bdd.H3.omf}, we get
	\EQN{
	\ddt\norm{\eta(\cdot,t)}_2
	&\leq \norm{\na^{\perp}\Del^{-1}T_{\ga}\eta \cdot\na \om_f}_2
	+\norm{\na^{\perp}\Del^{-1}T_{\ga}\om_g \cdot \na\om_f}_2\\
	&\lesssim_{M_1} \norm{\eta}_2\norm{\na \om_f}_6 + \norm{\na^{\perp}\Del^{-1}T_{\ga}\om_g}_{L^{\infty}(\supp(\om_f(\cdot,t)))}\norm{\na \om_f}_2\\
	&\leq C(\norm{\eta}_2+ \norm{\na^{\perp}\Del^{-1}T_{\ga}\om_g}_{L^{\infty}(\supp(\om_f(\cdot,t)))}),
	}
 	for some positive constant $C$ depending on $\norm{f}_{H^3}$,  $M$, and $M_1$. Then by Gr\"{o}nwall inequality, we have
	\EQ{\label{est.eta}
	\max_{0\leq t\leq 1}\norm{\eta(\cdot,t)}_2 \leq C(\norm{f}_{H^3}, M, M_1)
	\max_{0\leq t\leq 1}
	\norm{\na^{\perp}\Del^{-1}T_{\ga}\om_g}_{L^{\infty}(\supp(\om_f(\cdot,t)))}.
}
	
	 Using Lemma \ref{ker.H.lem} and \eqref{dist.supp.omfg}, we have for any $0\leq t\leq 1$ and $x\in \supp(\om_f(\cdot,t))$,
	\EQ{\label{est.ug}
	|\na^{\perp}\Del^{-1}T_{\ga}\om_g(x,t)|&=|H\ast\om_g(x,t)|\\
	&\lesssim \int_{|x-y|\geq \frac 12 R} \frac 1{|x-y|} |\om_g(y,t)| dy
	\lesssim R^{-1}\norm{g}_1\leq MR^{-1}. 
	}
	Finally, combining \eqref{est.H2}-\eqref{est.ug}, we can find $R=R(\ep,\norm{f}_{H^3}, M,M_1)>100C_0M$ sufficiently large such that
	\[
	\max_{0\leq t\leq 1}\norm{(\om_f-\td\om)(\cdot,t)}_{H^2}
	\leq C(\norm{f}_{H^3}, M,M_1)R^{-\frac 13}<\ep.
	\]	
	\end{proof}

	Now we are ready to prove the proposition.
	
	\noindent\textit{Proof of Proposition \ref{prop.patching.noncompact}.}
	Let $\om_{\leq n}$, $n\in \N$, be a smooth solution to
	\EQ{\label{defn.wn}
	\begin{cases}
	\pa_t \om_{\le n} +\na^{\perp}\Del^{-1}\om_{\le n}\cdot \na\om_{\le n} =0,\\
	\om_{\le n}|_{t=0}= \sum_{k=1}^n \om_{k0}(x-x_{k}).
	\end{cases}.
	}
	Our strategy is to construct a sequence $\{x_k\}_{k\in \N}$ of centers such that the following hold.
	\begin{enumerate}[(i)]
		\item\label{i} For each $j\in \N$, $\{\om_{\leq n}\}$ is Cauchy in $C([0,1];H^2(B(x_j,3C_0M)))$. 
		\item \label{ii} For any $n\in \N$,
		\[\label{disj.supp.wn}
		\supp(\om_{\leq n}(\cdot,t))\subset \bigcup_{j=1}^{\infty}B(x_j, 3C_0M).
		\]
		\item\label{iii} For any $n\in \N$ and $1\leq j\leq n$,
		\[
		\max_{0\leq t\leq 1}\norm{(\om_{\leq n}-\om_j)(\cdot,t)}_{H^2(B(x_j,3C_0M))} <\frac 1{2^{j+1}}.
		\]
	\end{enumerate}
	Then, the limit solution of $\{\om_{\leq n}\}$ becomes the desired one $\om$. 
	\bigskip
	
	\noindent\texttt{Step 1} Construction of the sequence $\{x_k\}_{k \in \N}$.
	 
	For each $j\in \N$, apply Lemma \ref{simple.patching.2} for $f=\om_{j0}$ and $\ep=\frac 1{2^{j+1}}$. Then, we can find $R_j>0$ such that for any $h\in H^2\cap L^1$ with
	\EQ{\label{h.ini}
	\norm{\om_{j0}}_1 &+\norm{h}_1 +\sup(\norm{\om_{j0}}_{\infty},\norm{h}_{\infty}) \leq M,\\
	&\dist(\supp(\om_{j0}), \supp(h))\geq R_j, 
	}
	where $M$ is given in \eqref{initial.data.M}, the solutions $\om$ and $\td\om_j$ to \eqref{main.eq} for the initial data $\om_{j0}+h$ and $\om_{j0}$, respectively, satisfy
	\EQ{\label{h.conclusion}
	\max_{0\leq t\leq 1}\norm{(\om-\td\om_j)(\cdot,t)}_{H^2(B(0,3C_0M))}<\frac 1{2^{j+1}},
	}
	and
	\EQ{\label{supp.h}
	\supp(\om(\cdot,t))\subset B(0,3C_0M) \cup B(\supp(h), 2C_0M).	
}
Here, \eqref{supp.h} is an easy consequence of \eqref{condition.decomp.f} and \eqref{condition.decomp.g}. 
	
	We find $\{x_n\}$ inductively. Indeed, we can relax the conditions on $\{x_n\}$ as follows; for any $n\geq 2$ in $\N$ with $x_1=0$, 
\begin{enumerate}[(a)]
		\item\label{choice.x1.claim} 
		$x_n$ is located at a far distance from previously chosen points
		\EQN{
			|x_n-x_l|>\sum_{i=1}^n R_i + 10C_0M + 2^n, 	\quad\forall 1\leq j< n,
		}	
		\item\label{choice.x2.claim} 
		A smooth solution $\om_{\leq n}$ to \eqref{defn.wn} satisfies 
		\EQN{
			\supp(\om_{\leq n}(\cdot,t))\subset \bigcup_{j=1}^n B(x_j,3C_0M), \quad \forall 0\leq t\leq 1.
		}
		\item\label{choice.x3.claim}
		Denoting $B(x_j, 3C_0M)$ by $B_j$, 	
		\EQN{
			\max_{0\leq t\leq 1}\norm{(\om_{\leq n}-\om_{\leq n-1})(\cdot,t)}_{H^2(\bigcup_{j=1}^{n-1} B_j)}<\frac 1{2^n}. 
		}	
\end{enumerate}
	
Then, the requirements \eqref{i} and $\eqref{ii}$ easily follow from \eqref{choice.x3.claim} and \eqref{choice.x2.claim}, respectively. We can also check that \eqref{choice.x1.claim} implies \eqref{iii}. For each $n\in \N$ and $1\leq j\leq n$, plug 
	\EQ{\label{given.h}
	h(x)=\sum_{\substack{k=1\\k\neq j}}^{n}\om_{k0}(x-x_k+x_j)
}
	into \eqref{h.ini}. We can easily see that \eqref{h.ini} holds true
	\EQN{
		\norm{\om_{j0}}_1 +\norm{h}_1 +\sup(\norm{\om_{j0}}_{\infty},\norm{h}_{\infty}) 
		\leq \sum_{k=1}^n\norm{\om_{k0}}_1 +\sup_{1\leq k\leq n}\norm{\om_{k0}}_{\infty} 
		\leq M, 
	}	
	and
	\EQN{
		\dist(\supp(\om_{j0}), \supp(h))
		&=\dist(\supp(\om_{j0}(\cdot-x_j)), \supp(h(\cdot-x_j)))\\
		&\geq \inf_{\substack{1\leq k\leq n\\k\neq j}}\dist(B(x_j,1), B(x_k,1))\\
		&\geq \inf_{\substack{1\leq k\leq n\\k\neq j}}
		|x_j-x_k|-2
		\geq R_j.
	}
	Therefore, using the translation invariant property of \eqref{main.eq}, we have
	\[
	\max_{0\leq t\leq 1}\norm{\om_{\leq n}(\cdot+x_j,t)-\td\om_j(\cdot,t)}_{H^2(B(0,3C_0M))}
	<\frac 1{2^{j+1}},
	\]
	which follows \eqref{iii}.

	Now, we choose $\{x_j\}$ satisfying \eqref{choice.x1.claim}-\eqref{choice.x3.claim} by induction. At the end of each inductive step, we also find $\td R_n\geq \td R_{n-1}$ satisfying the following condition
	\begin{enumerate}[(a)]\setcounter{enumi}{3}
		\item \label{condi.tdR}
		For any $g\in H^2\cap L^1$ with
		\EQ{\label{condition.for.g}
			&\norm{\sum_{j=1}^n\om_{j0}(\cdot-x_j)}_1 +\norm{g}_1 +\sup\left(\norm{\sum_{j=1}^n \om_{j0}(\cdot-x_j)}_{\infty},\norm{g}_{\infty}\right) \leq M,\\
			&\dist\left(\supp\left(\sum_{j=1}^n \om_{j0}(\cdot-x_j)\right), \supp(g)\right)\geq \td R_{n}, 
		}
		the solution $\om$ to \eqref{main.eq} for the initial data $\sum_{j=1}^n \om_{j0}(x-x_j)+g$ satisfies
		\[
		\supp(\om(\cdot,t)) \subset \left(\bigcup_{j=1}^n B_j\right) \bigcup B(\supp(g),2C_0M),
		\quad\forall 0\leq t\leq 1
		\]
		and
\EQ{\label{conv.H2}
		\max_{0\leq t\leq 1}\norm{(\om-\om_{\leq n})(\cdot,t)}_{H^2(\bigcup_{k=1}^n B_k)}<\frac 1{2^{n+1}}.
}
	\end{enumerate}

	Set $x_1=0$ and $\td R_1 =R_1$. We first choose $x_2$ satisfying
	\[
	|x_2-x_1| > \sum_{i=1}^{2} R_i + 10C_0M+ 2^{2} +\td R_{1}. 
	\]
	Clearly, \eqref{choice.x1.claim} for $n=2$ is obtained. Also, $j=1$ and $h=w_{20}(x-x_2)$ satisfies \eqref{h.ini}, which implies \eqref{choice.x2.claim}-\eqref{choice.x3.claim} for $n=2$. Here, we use $\om_{\leq 1}=\om_1=\td \om_1$. 
	
	The choice of $\td R_2\geq \td R_1 = R_1$ satisfying \eqref{condi.tdR} for $n=2$ follows from Lemma \ref{simple.patching.2}; apply it to $f=\om_{\leq 2}|_{t=0}$ and $\ep=\frac 1{2^3}$. 
		
	Assume that $\{x_j\}_{j=1}^n$ and $\td R_{n}$ are given and satisfy \eqref{choice.x1.claim}-\eqref{condi.tdR}.
Then, we pick $x_{n+1}$ such that 
\[
|x_{n+1}-x_j|>	\sum_{i=1}^{n+1} R_i + 10C_0M+ 2^{n+1} +\td R_{n} , 
\quad\forall j=1,\cdots,n.
\]
which follows \eqref{choice.x1.claim}. To achieve \eqref{choice.x2.claim} and \eqref{choice.x3.claim} for $n+1$, we observe that $g=\om_{(n+1)0}(x-x_{n+1})$ satisfies \eqref{condition.for.g},
\EQN{
	\norm{\sum_{j=1}^n\om_{j0}(\cdot-x_j)}_1 +\norm{g}_1 &+\sup\left(\norm{\sum_{j=1}^n \om_{j0}(\cdot-x_j)}_{\infty},\norm{g}_{\infty}\right) \\
	&\leq \sum_{j=1}^{\infty}\norm{\om_{j0}}_1  
	+\sup_j\norm{\om_{j0}}_{\infty}
	\leq M
}
and
\EQN{
	\dist\left(\supp\left(\sum_{j=1}^n \om_{j0}(x-x_j)\right), \supp(g)\right)
	&\geq \inf_{1\leq j\leq n}\dist(B(x_j,1), B(x_{n+1},1))\\
	&\geq \inf_{1\leq j\leq n}|x_{n+1}-x_j| - 2  
	\geq \td R_{n}.
}	
Then by \eqref{condi.tdR} for $n$, the conditions \eqref{choice.x2.claim} and \eqref{choice.x3.claim} for $n+1$ hold; we have 
\[
\supp(\om_{\leq n+1}(\cdot,t)) 
\subset \bke{\bigcup_{j=1}^{n}B_j }\cup B(x_{n+1}, 2C_0M +1)
\subset \bigcup_{j=1}^{n+1}B_j
\]
and
\[
\max_{0\leq t\leq 1} \norm{(\om_{\leq n+1} - \om_{\leq n})(\cdot,t)}_{H^2(\bigcup_{k=1}^{n}B_k)} <\frac 1{2^{n+1}}. 
\]

Applying again Lemma \ref{simple.patching.2} for $f=\om_{\leq n+1}|_{t=0}= \sum_{j=1}^{n+1} \om_{j0}(x-x_j)$ and $\ep=\frac 1{2^{n+2}}$, we can find $\td R_{n+1}\geq \td R_n$ satisfying \eqref{condi.tdR}. Therefore, we have \eqref{choice.x1.claim}-\eqref{condi.tdR} at $(n+1)$th step, so that they hold true for any $n\geq 2$.

		\bigskip
	
	\noindent\texttt{Step 2.} Check the required conditions. 
	
	By the condition \eqref{i}, $\{\om_{\leq n}\}$ is Cauchy in $C([0,1];H^2(B(x_j, 3C_0M)))$ for each $j\in \N$. On the other hand, by Lemma \ref{simple.patching}, for each $j\in \N$ and $k\geq 2$, $\{\om_{\leq n}\}$ is uniformly bounded in $C([0,1];H^k(B(x_j, 3C_0M)))$, so that $\{\om_{\leq n}\}$ is Cauchy even in $C([0,1];H^k(B(x_j, 3C_0M)))$. This implies that for each $0\leq t\leq 1$, we have a pointwise limit solution 
	\[
	\om(x,t) =
	\begin{cases}
	\lim_{n\to\infty}\om_{\leq n}(x,t) 	&x\in \bigcup_{j=1}^{\infty}B(x_j,3C_0M)\\
	0							    &\text{otherwise.}
	\end{cases}
	\]
Obviously, $\om(\cdot,t)\in C^{\infty}$ and $\om$ satisfies \eqref{supp.om.dis} and \eqref{local.behavior.prop} by the conditions \eqref{ii} and \eqref{iii}. Furthermore, $\om\in C([0,1];L^1(\R^2)\cap L^{\infty}(\R^2))$. This is because for any $0\leq t\leq 1$, we have
	 \EQN{
	 	\norm{\om(\cdot,t)}_{1}
	 	&=\sum_{j=1}^{\infty}\norm{\om(\cdot,t)}_{L^1(B_j)}
	 	=\sum_{j=1}^{\infty} \lim_{n\to\infty}\norm{\om_{\leq n}(\cdot,t)}_{L^1(B_j)}
	 	=\sum_{j=1}^{\infty} \norm{\om_{j0}}_{1}
	 	= \norm{\om_0}_1
 		 }
and    
	 \EQN{
	 	\norm{\om(\cdot,t)}_{\infty}
	 	&=\sup_j\norm{\om(\cdot,t)}_{L^{\infty}(B_j)}
	 	=\sup_j\lim_{n\to\infty}\norm{\om_{\leq n}(\cdot,t)}_{L^{\infty}(B_j)}\\
	 	&=\sup_j\norm{\om_{j0}}_{\infty}=\norm{\om_0}_{\infty}.
	 }

	 Finally, we prove that the limit solution $\om$ is the unique classical solution to \eqref{main.eq} for the initial data 
	\[
	\om|_{t=0}(x) = \sum_{j=1}^{\infty}\om_{j0}(x-x_j). 
	\]
	We first show that the limit solution $\om$ solves \eqref{main.eq} in the sense of 
	\begin{align}
	\om(x,t) = \om_0(x) &- \int_0^t (\na^{\perp}\Del^{-1}T_{\ga}\om\cdot\na\om)(x,s) ds, 
	\quad\forall (x,t)\in \R^2\times (0,1).
	\label{main.eq.inte}
	\end{align}	
At $t=0$, it is apparent that the limit solution is same with $\om_0$. Since $\om_{\leq n}$ solves \eqref{main.eq.inte} with $\om_0 = \sum_{j=1}^{n}\om_{j0}(\cdot-x_j)$ for any $n\in \N$, it is enough to prove the uniform convergence $\na^{\perp}\Del^{-1}T_{\ga}\om_{\leq n}\to \na^{\perp}\Del^{-1}T_{\ga}\om$ on each $B(x_j,3C_0M)\times [0,1]$, $j\in \N$. 
For notational simplicity, we suppress the dependence on the variable $t$, if it's not needed. Fix $j\in \N$. For $n>j$ and $x\in B(x_j,3C_0M)=B_j$, we have 
	\EQN{
	|(\Del^{-1}\na^{\perp}T_{\ga}(\om_{\leq n}-\om)(x)|
	&\leq \int |H(x-y)||(\om_{\leq n}-\om)(y)| dy\\
	&=\left(\sum_{\substack{m=1\\m\neq j}}^{n}\int_{B_m}+\int_{B_j}+\sum_{l=n+1}^{\infty}\int_{B_l}\right) |H(x-y)||(\om_{\leq n}-\om)(y)| dy\\
	&=I_1^n + I_2^n +I_3^n.
	}
	
	By the choice of the centers, we have for any $x\in B_j$ and $y\in B_m$, $m\neq j$,
	\[
	|x-y|\geq |x_j-x_m|-6C_0M \geq 2^{\max(j,m)}.
	\]
	This implies that $I_1^n$ converges to 0, as $n$ goes to infinity; for $x\in B_j$,
	\EQN{
	I_1^n&\lesssim \sum^{n}_{\substack{m=1\\m\neq j}}\int_{B_m} \frac 1{|x-y|}|(\om_{\leq n}-\om)(y,t)| dy
	\leq \sum_{m=1}^{n}2^{-m}\norm{(\om_{\leq n}-\om)(\cdot,t)}_{L^1(B_m)}\\
	&\lesssim\sum_{m=1}^{n}2^{-m}\norm{\om_{\leq n}-\om}_{C([0,1];L^{\infty}(B_m))} \to 0, \quad\text{ as } n\to \infty.
	}

	In a similar way, $I_3^n$ approaches to 0, as $n$ goes to infinity;
	\[
	I_3^n \lesssim \sum_{l=n+1}^{\infty}\int_{B_l} \frac 1{|x-y|}|\om(y)| dy\\
	\leq \sum_{l=n+1}^{\infty}2^{-l}\norm{\om_0}_{1} \to 0, \quad\text{ as } n\to \infty. 
	\]
	
	Finally, since $|x-y|\leq |x-x_j|+|y-x_j|\leq 6C_0M$, we obtain
	\[
	I_2^n\lesssim_{M} 
	\max_{0\leq t\leq 1}\norm{(\om_{\leq n}-\om)(\cdot,t)}_{L^{\infty}(B_j)} \to 0, \quad\text{as }n\to \infty.
	\]
	Therefore, we get the uniform convergence of $\na^\perp \De^{-1}T_\ga \om_{\le n}$ and hence $\om$ solves $\eqref{main.eq}$ in the sense of \eqref{main.eq.inte}. Using the equation, we can improve the regularity of the solution in time, so that $\om$ is a classical solution to \eqref{main.eq}. 
	
	For the uniqueness of the classical solution, let $\bar{\om}$ be another classical solution to \eqref{main.eq} for the same initial data. Note that the statement in Lemma \ref{simple.patching.2} holds also for a classical solution $\om$ for initial data $f+g$ where $g\in C^\infty(\R^2)\cap L^1(\R^2)$.
Then, in the same way of obtaining \eqref{conv.H2}, we have
	\EQN{
	\max_{0\leq t\leq 1}\norm{(\om-\om_{\leq n})(\cdot,t)}_{H^2(\cup_{j=1}^n B_j)} <\frac 1{2^n}\\
	\max_{0\leq t\leq 1}\norm{(\bar\om-\om_{\leq n})(\cdot,t)}_{H^2(\cup_{j=1}^n B_j)} <\frac 1{2^n}.
	}
This follows from that $g=\sum_{j=n+1}^{\infty}\om_{j0}(\cdot-x_j)$ satisfies \eqref{condition.for.g} for the same $M$, $f$, and $\ep$ in the construction of $\td R_{n+1}$.  
Therefore, we have $\om=\bar\om$. In other words, the uniqueness of the classical solution holds. 
	This completes the proof.
	\hfill$\square$

\section{Proof of Theorem \ref{thm.noncpt}}\label{sec.non.comp}
In this section, combining the results obtained in the previous sections, we finally construct a non-compactly supported perturbation for the strong ill-posedness of \eqref{main.eq} in the critical Sobolev space.

\

\noindent\textit{Proof of Theorem \ref{thm.noncpt}.} Recall the family of initial data $\td g_A$ in Remark \ref{local.sol.family}. By its construction, for fixed $0<\ga\leq \frac 12$ and $0<\ep<1$, we can find a sequence $\{A_j\}$ such that for any $j\in \N$, $\zeta_j = \td g_{A_j}$ satisfies $\supp(\zeta_j)\subset B(0,1)$ and
 \EQ{\label{xi.smallness}
 \norm{\zeta_j}_1 + \norm{\zeta_j}_{\infty}+\norm{\na \zeta_j}_{2}<\frac {\ep}{2^j},
}
 and the smooth solution $\td\om_j$ to \eqref{main.eq} with initial data $\zeta_j$ achieves 
 \EQ{\label{largeness.omj}
 \norm{\na \td\om_j(\cdot,t_j)}_2 > j
}
 for some $t_j$ which converges to $0$ as $j\to \infty$. 
  
Since the solution to \eqref{main.eq} is translation-invariant, in the case of $\supp(a)\subset B(0,1)$ up to translation, we can apply Proposition \ref{prop.patching.noncompact} to $\om_{10}=a$ and $\om_{j0}=\zeta_j$ for $j\geq 2$. Then, we have a sequence $\{x_j\}_{j\in \N}$ of centers with $x_1=0$ such that for the initial data 
 \[
 \om_0(x) = a(x-x_1) + \sum_{j=2}^{\infty}\zeta_{j}(x-x_j) =: a(x)+\zeta(x)
 \]
 we have a unique classical solution $\om$ to \eqref{main.eq} and the solution satisfies $\om(\cdot,t)\in C^{\infty}(\R^2)$ for any $0\leq t\leq 1$, $\om\in C([0,1];L^1(\R^2)\cap L^{\infty}(\R^2))$, and
 \EQ{\label{om.close.omj}
 \max_{0\leq t\leq 1}\norm{(\om-\om_j)(\cdot,t)}_{H^2(B(x_j,3C_0M))}<1
}
 for sufficiently large $j$. Here, $\om_j$ is a smooth solution to \eqref{main.eq} for the initial data $\zeta_j(x-x_j)$, $C_0$ is the constant defined in Proposition \ref{prop.patching.noncompact}, and $M>1$ is a bound of the initial data in the sense of
 \[
 1+\norm{a}_{H^1}^2+\norm{a}_1 +\norm{a}_{\infty} +\norm{\zeta}_{H^1}^2+\norm{\zeta}_1+\norm{\zeta}_{\infty}\leq M.
 \] 
Note that $\om_j$ for any $j\in \N$ satisfies
\[
\supp(\om_j)\subset B(x_j,3C_0M),\quad
\om_j(x,t)=\td\om_j(x-x_j,t). 
\]
It is easy to see that $\zeta \in C^\infty(\R^2)$ because of $\zeta_j \in C_c^\infty(B(0,1))$ and $|x_j-x_k|\gg 1$ for $j\neq k$.
By \eqref{xi.smallness}, we also get
 \[
 \norm{\zeta}_{\dot{H}^1(\R^2)}+\norm{\zeta}_1 +\norm{\zeta}_{\infty}
 \leq \sum_{j=2}^{\infty} \norm{\na\zeta_j}_{2}+\norm{\zeta_j}_1 +\norm{\zeta_j}_{\infty}<\ep.
 \]
On the other hand, \eqref{largeness.omj}, \eqref{om.close.omj}, and $\supp(\om_j(\cdot,t))\subset B(x_j,3C_0M)$, $0\leq t\leq 1$, implies that
 \EQN{
 	\norm{\om(\cdot,t_j)}_{\dot{H}^1(B(x_j,3C_0M))} 
 	&\geq \norm{\om_j(\cdot,t_j)}_{\dot{H}^1(B(x_j,3C_0M))} -\norm{(\om-\om_j)(\cdot,t_j)}_{\dot{H}^1(B(x_j,3C_0M))}\\
 	&\geq \norm{\td\om_j(\cdot,t_j)}_{\dot{H}^1(\R^2)} - \max_{0\leq t\leq 1}\norm{(\om-\om_j)(\cdot,t)}_{\dot{H}^1(B(x_j,3C_0M))}\\
 	&> j - 1.
 }
Therefore, the constructed perturbation $\zeta$ satisfies all requirements in Theorem \ref{thm.noncpt}. If $\supp(a)\not\subset B(0,1)$ up to translation, we slightly modify the proof of the Proposition and obtain the same conclusion. 

\hfill$\square$


\section{The compact case}\label{sec.comp}

In this section, we prove Theorem \ref{thm.cpt}, the compact case. Unlike the non-compact case, a large distance between local solutions cannot be used in order to minimize their interactions and make a global solution locally behave like local ones. For this reason, we adopt a different scheme; use the smallness in $L^1$-norm of the tail part of a global solution.

The following proposition describes a simple scenario of patching.

\begin{proposition}\label{prop.cpt}  
	Suppose that $f\in C_c^{\infty}(\R^2)$ satisfies
		\EQ{\label{prop.f}
		\supp(f)&\subset \{(x_1,x_2)\in \R^2: x_1\leq -2R_0\} \quad\text{ for some }R_0>0,\\
		&f(x_1,x_2) = -f(x_1,-x_2) \qquad\qquad\forall (x_1,x_2)\in \R^2. 
		}
	Then, for any $0<\ep_0<\frac{R_0}{100}$, we can find $\del=\del(f,\ep_0,R_0)>0$, $t_0=t_0(f,\ep_0,R_0)\in (0,\ep_0)$, and $g=g(f,\ep_0,R_0)\in C_c^{\infty}(B(0,\ep_0))$ such that the following holds.
	\begin{enumerate}[(i)]
		\item $g$ satisfies 
		\EQN{
		&\norm{g}_{\dot{H}^1} + \norm{g}_{\infty}+ \norm{g}_1 + \norm{g}_{\dot{H}^{-1}} < \ep_0\\
		&g(x_1,x_2) = -g(x_1,-x_2), \qquad\forall (x_1,x_2)\in \R^2.
		}
		\item For any given $h\in C_c^{\infty}(\R^2)$ with
		\EQ{\label{condi.h}
		\supp(h)\subset \{(x_1,x_2):x_1\geq R_0\},\quad
		\norm{h}_1 +\norm{h}_{\infty}\leq \del, 	
	}	
	the smooth solution $\om$ to \eqref{main.eq} for the initial data $\om|_{t=0}= f+g+h$ has a decomposition
	\[
	\om = \om_f+\om_g+\om_h,\quad \text{ on }\R^2\times [0,t_0]
	\]
	such that
		\EQ{\label{supp.omfgh}
		\supp(\om_f(\cdot,t))&\subset B(\supp(f),\frac 18R_0),\\
		\supp(\om_g(\cdot,t))&\subset B(0,\ep_0+\frac 18R_0),\\
		\supp(\om_h(\cdot,t))&\subset B(\supp(h), \frac 18R_0), \quad\forall 0\leq t\leq t_0
		}
		and
		\EQ{\label{norm.inflation.omg}
		\norm{\om_g(\cdot,t_0)}_{\dot{H}^1}&>\frac 1{\ep_0}. 
		}	
	\end{enumerate}
\end{proposition}

To prove this proposition, we need some preliminary lemmas. The first lemma is about the finite time propagation.

\begin{lemma}\label{supp.sol.lemma} Let $\Om$ be a smooth solution to 
	\EQ{\label{eqn.W.BE}
	\begin{cases}
	\pa_t \Om+\na^{\perp}\Del^{-1}T_{\ga}\Om \cdot \na \Om + \left(B+ E-C\right) \cdot\na \Om=0\\
	C(t) = (-\pa_2\De^{-1}T_{\ga}\Om(0,0,t),\ 0)^\intercal \\
	\Om|_{t=0} = \Om_0
	\end{cases}
}
	where $B$, $E$, and $\Om_0$ are smooth functions satisfying 
	\begin{itemize}
		\item 
		\[
		\norm{\Om_0}_\infty \leq B_0, \quad \text{for some } B_0>0,
		\]
		\EQ{\label{supp.in.bR}
		\supp(\Om_0)\subset B(0,R),\quad \text{for some } R>0,
	}
		\item $B$ and $E$ are divergence-free
		\[
		\na \cdot B = \na \cdot E =0.
		\]
		\item For some positive numbers $B_1$ and $B_2$,
		\[
		|B(y,t)|\leq B_1|y|, \quad |E(y,t)|\leq B_2 |y|^2, \quad\forall (y,t)\in \R^2\times [0,1].
		\]
	\end{itemize}
	
Then, we can find $R_0>0$ and $0<t_0<1$ both depending only on $B_0$, $B_1$ and $B_2$ such that if $0<R\leq R_0$,
a characteristic line $\Phi$ which solves 
\EQ{\label{eq.Phi.BE}
\begin{cases}
\pa_t \Phi(y,t) 
= (\na^{\perp}\Del^{-1}T_{\ga}\Om+B+E-C)(\Phi(y,t),t)\\
\Phi(y,0) = y
\end{cases}
}
satisfies
\[
|\Phi(y,t)|\le 2R, \quad\forall |y|\le R, \ t\in [0,t_0].
\]
In particular, the solution $\Om$ satisfies
\[
\supp(\Om(\cdot,t))\subset B(0,2R), \quad\forall 0\leq t\leq t_0. 
\]
\end{lemma}

\begin{proof}

From \eqref{eq.Phi.BE}, we obtain
\EQ{\label{eq.abs.Phi}
\pa_t |\Phi(y,t)|\leq 2\norm{\na^{\perp}\Del^{-1}T_{\ga}\Om}_{\infty} + B_1|\Phi(y,t)| + B_2|\Phi(y,t)|^2.
}
By using $L^p$-norm preservation and \eqref{supp.in.bR}, we have
\[
 \norm{\na^{\perp}\Del^{-1}T_{\ga}\Om}_{\infty}
 \lesssim\norm{\Om}_1^{\frac 12}\norm{\Om}_{\infty}^{\frac 12}
 \lesssim R\norm{\Om_0}_{\infty}\leq RB_0.
\]
Combining with \eqref{eq.abs.Phi}, we can find $t_0>0$ and $R_0>0$ such that if $0<R\leq R_0$,  
\[
|\Phi(y,t)|\leq 2R, \quad\forall |y|\le R, \ t\in [0,t_0].
\]

Furthermore, using the characteristic, \eqref{eqn.W.BE} can be written as $\Om(\Phi(y,t),t)=\Om_0(y)$, so that
\[
\supp(\Om(\cdot,t))\subset \Phi(\supp(\Om_0),t).
\]
Then, it easily follows that $\supp(\Om(\cdot,t))\subset B(0,2R)$ for any $0\leq t\leq t_0$.   
\end{proof}

\bigskip

Recall the definition of $g_A$ in \eqref{defn.gA}. This family of initial data was used in order to create large Lagrangian deformation. Now, we redefine $g_A$ when $\ga = \frac 12$ by 
\EQ{\label{defn.gA12}
g_A (x)= \frac 1{\ln\ln\ln A}\frac 1{\sqrt{\ln\ln A}}\sum_{A\leq j<A\ln A} \frac 1{\sqrt{j}}\rho(2^j x),
}
where $\rho$ is given as in \eqref{defn.rho}.
In the case of $0<\ga<\frac 12$, we use the same $g_A$ in \eqref{defn.gA}. 

Then, $g_A$ satisfies
\begin{itemize}
\item $\supp(g_A)\subset B(0,2\cdot 2^{-A})$.
\item 
\[
\norm{g_A}_1 \lec 2^{-2A}, \quad
\norm{g_A}_\infty \lec \frac 1{A^\ga}, \quad
\norm{g_A}_{\dot{H}^{-1}}\lec 2^{-A}, \quad
\norm{\na g_A}_2 \lec \frac 1{\ln\ln\ln A}.
\]

\item 
\[
 \int_{z_1>0,z_2>0}\frac{1}{|z|^2} \ln^{-\ga}\left(e + \frac 1{|z|}\right) e^{-|z|^4} g_A(z)dz 
\gtrsim \frac {\sqrt{\ln\ln A}}{\ln\ln\ln A}. 	
\]
\end{itemize}

\medskip
From this newly redefined family $\{g_A\}$, we extract a sequence of local initial data and sequentially patch them to the given initial data $a$, given in the statement of Theorem \ref{thm.cpt}. Here, this patched one becomes a desired perturbed initial data. To this end, the next two lemmas confirm that the large Lagrangian deformation created by a current initial data will not be destroyed even in the presence of the previously chosen ones. In the first lemma, we estimate $\cR_{ii}T_\ga (g_A\circ\phi_A)= \De^{-1}\pa_{ii}T_\ga (g_A\circ\phi_A)$ in $L^\infty$-norm, which is a key ingredient of Lemma \ref{lld.cpt}. The proof can be obtained by a slight modification of the one for Lemma 3.2 in \cite{BL15}.

\begin{lemma}\label{small.reisz.lem} Let $\{g_A\}$ be a family of functions defined as in \eqref{defn.gA} for $0<\ga<\frac 12$ and \eqref{defn.gA12} for $\ga=\frac 12$. 
	Suppose that $\phi_A=(\phi_A^1,\phi_A^2):\R^2 \to \R^2$ is a bi-Lipschitz function such that
	\begin{itemize}
		\item  $\phi_A(0)=0$.
		\item $\phi_A^1(y_1,-y_2) = \phi_A^1 (y_1,y_2) $ and $\phi_A^2(y_1,-y_2) = -\phi_A^2 (y_1,y_2)$.
		\item For some integer $m_A \ge 1$, 
		\EQ{\label{defm.norm}
		\norm{D\phi_A}_{L^\infty(|y|\leq 4\cdot 2^{-A})} \leq 2^{m_A}, \quad \norm{D(\phi_A^{-1})}_{L^\infty(|y|\leq 2\cdot 2^{-A})} \leq 2^{m_A}. 
		}
		\item $|\det(D\phi_A)| = |\det(D(\phi_A^{-1}))| =1$.
		\item If $|\phi_A(y)|\leq 2\cdot 2^{-A}$, then $|y|\leq 4\cdot 2^{-A}$.
	\end{itemize}
	Then, we have
	\EQ{\label{small.Riesz}
	\norm{\cR_{11}T_\ga (g_A\circ\phi_A)}_\infty +\norm{\cR_{22}T_\ga (g_A\circ\phi_A)}_\infty \lesssim_{\ga} \frac {2^{m_A}}{\sqrt{\ln\ln A}}.
	}
\end{lemma}

\begin{proof}
Recall the definition of $g_A$,
\EQN{
	g_A(y)
	=
	\begin{cases}
		C_A\sum_{a_A\leq j < b_A}\frac 1{j^{\ga}}\rho(2^j y), &0<\ga<\frac 12\\[10pt]
		\frac 1{\ln\ln\ln A}\frac 1{\sqrt{\ln\ln A}}\sum_{A\leq j<A\ln A} \frac 1{\sqrt{j}}\rho(2^j y), &\ga=\frac 12
	\end{cases}
}
where $C_A=\frac 1{\sqrt{\ln A}}\frac 1{\ln\ln A}$, $a_A=A^{\frac 1{1-2\ga}}$, and $b_A=(A+\ln A)^{\frac 1{1-2\ga}}$. Here, $\rho$ is an odd function in both variables and satisfies  $\frac 12\leq |x|\leq 2$ for $x\in \supp(\rho)$. (See \eqref{defn.rho})

\medskip

We first consider $\cR_{ii}T_\ga (\rho_j\circ \phi_A)$ for $j\geq A$, where $\rho_j =\rho(2^j\cdot)$. For the convenience, we drop the index $A$ in $g_A$, $\phi_A$ and $m_A$ below. Denote the kernel for the operator $\cR_{ii}T_\ga$ by $K_{ii}$ for $i=1,2$ and fix $y\in \R^2\setminus\{0\}$ with $|y|\sim 2^{-l}$ for some $l$. Note that the kernel $K_{ii}$, $i=1,2$, can be obtained by taking a weak derivative to the kernel $H_1$ and $H_2$ of $\na^\perp\De^{-1} \ln^{-\ga}(e-|\na|)$ and $\na^\perp\De^{-1} \ln^{-\ga}(e-\De)$, respectively, which are given in \eqref{H1} and \eqref{H2}. Then, we can easily see from \eqref{est.H} that
$|K_{ii}(y)|\lec \frac 1{|y|^2}$ for $y\neq 0$. 
\bigskip 

\noindent\textbf{Case 1.} $2^j \ll 2^{l-m}$. 

By the assumption on $\phi$, for $x$ with $|\phi(x)|\leq 2\cdot 2^{-A}$, we have $|x|\leq 4\cdot 2^{-A}$. Then, using $\phi(0)=0$ and \eqref{defm.norm},  $x$ with $2^{-j-1}\leq |\phi(x)|\leq  2^{-j+1}$ satisfies 
\EQ{\label{y-z}
2^{-j+m}\gtrsim |x| \gtrsim 2^{-j-m}. 
}
Therefore, if $y$ and $z$ satisfy $\phi(y-z)\in \supp(\rho_j)$, we have 
$2^{-j-1}\leq |\phi(y-z)|\leq 2^{-j+1}$ and hence $2^{-l}\ll 2^{-j-m}\lec |y-z| \lec 2^{-j+m} $. Combining with $|y|\sim 2^{-l}$, for such $y$ and $z$, we get
\[
2^{-j-m}\lec |z| \lec 2^{-j+m}. 
\]
Now, we estimate $\cR_{ii}T_\ga (\rho_j\circ \phi)$ for $i=1,2$.
\EQN{
|\cR_{ii}T_\ga (\rho_j\circ\phi)(y) |
&= \left|\int (\rho_j\circ\phi)(y-z) K_{ii}(z)dy\right|\\
&\le \int_{2^{-j-m}\lesssim |z| \lesssim 2^{-j+m}} 
|(\rho_j\circ\phi)(y-z)-(\rho_j\circ\phi)(-z)| |K_{ii}(z)| dy\\
&\lec |y|\norm{\na (\rho_j\circ\phi)}_\infty 
\int_{2^{-j-m}\lesssim |z| \lesssim 2^{-j+m}} \frac 1{|z|^2} dz\\
&\lec 2^{-l + m + j} m.
}
In the first inequality, we use $\phi(y-z)\in \supp(\rho_j)$ and 
\[
\cR_{ii}T_\ga (\rho_j\circ\phi)(0)
=\int_{c\leq |z|\leq C} (\rho_j\circ\phi)(-z) K_{ii}(z) dz =0
\]
for any arbitrary constants $0<c<C<+\infty$.
This is because $\phi^1$ and $K_{ii}$ for $i=1,2$ are even in $z_2$, while $\phi^2$, and $\rho$ are odd in $z_2$. 

\medskip

\noindent\textbf{Case 2.} $2^j \gg 2^{l+m}$

By \eqref{y-z} with $2^{-l}\gg 2^{-j + m}$, we have $|z|\sim 2^{-l}$ when $\phi(y-z) \in \supp(\rho_j)$ and $|y|\sim 2^{-l}$. This implies that for $i=1,2$
\[
|\cR_{ii}T_\ga (\rho_j\circ\phi)(y) |
\le \norm{K_{ii}}_{L^\infty(|y|\sim 2^{-l})} \norm{\rho_j\circ \phi}_1
\lec 4^{l-j} .
\]

	\medskip
	
\noindent\textbf{Case 3.} $	2^{l-m}\lec 2^j \lec 2^{l+m}$	

\[
\norm{\cR_{ii}T_\ga (\rho_j\circ \phi)}_\infty
\lec \norm{\rho_j\circ \phi}_2^\frac12 \norm{\na (\rho_j\circ \phi)}_\infty^\frac12
\lec \norm{\rho_j}_2^\frac12 \norm{\na \rho_j}_\infty^\frac12 2^{\frac{m}2}
\lec 2^{\frac{m}2}.
\]

\medskip

Combining all the cases, we have
\[
\sum_j \norm{\cR_{ii}T_\ga (\rho_j\circ \phi)}_\infty 
\lec 2^{\frac{m}2} m  + m \lec 2^{\frac{m}2} m, \quad i =1,2.
\]
Then, \eqref{small.Riesz} easily follows. 
\end{proof}

In the following lemma, $f$ represents the previously chosen initial data together with the given one. 
Then, it says that we can always find a current initial data in the family $\{g_A\}$ such that the deformation matrix for the patched initial data can still be large as desired. In other words, the current initial data still creates the large Lagrangian deformation even in the presence of the previously chosen one. 
\begin{lemma}\label{lld.cpt}
Suppose that $f$ satisfies \eqref{prop.f}. Let $\om$ be a smooth solution to
	\[
	\begin{cases}
	\pa_t \om+\na^{\perp}\Del^{-1}T_{\ga}\om \cdot \na \om =0\\
	\om|_{t=0} = f+ g_A.
	\end{cases}
	\]
Then, a characteristic line $\phi$ which solves
\EQ{\label{eq.Phi.12}
\begin{cases}
\pa_t \phi(x,t) = \na^{\perp}\Del^{-1}T_{\ga}\om(\phi(x,t),t)\\
\phi(x,0)=x
\end{cases}
}	
satisfies
	\EQ{\label{lld.cpt.ineq.12}
	\max_{0\leq t\leq \frac 1{\ln\ln\ln A}}\norm{D\phi(\cdot,t)}_{L^\infty(B(0,10\cdot 2^{-A}))}>\ln^{\frac 14}\ln\ln\ln A
}
for sufficiently large $A$. 
\end{lemma}

\begin{proof}
Suppose that \eqref{lld.cpt.ineq.12} doesn't hold true. i.e., 
\[
\max_{0\leq t\leq \frac 1{\ln\ln\ln A}}\norm{D\phi(\cdot,t)}_{L^\infty(|x|\leq 10\cdot 2^{-A})}\leq \ln^{\frac 14}\ln\ln\ln A.
\]
 First, we decompose the solution $\om$ into $\om_f$ and $\om_g$, where $\om_g$ solves
	\[
	\begin{cases}
	\pa_t \om_g+\na^{\perp}\Del^{-1}T_{\ga}\om_g \cdot \na \om_g+\na^{\perp}\Del^{-1}T_{\ga}\om_f \cdot \na \om_g =0\\
	\om_g|_{t=0} = g_A.
	\end{cases}
	\]
	Since both $f$ and $g_A$ are odd in $x_2$, so are $\om$ and $\om_g$. Also, we have 
	\[
	\phi_1(x_1,-x_2,t) = \phi_1(x_1,x_2,t), \quad
	\phi_2(x_1,-x_2,t) = -\phi_2(x_1,x_2,t)
	\]
	and therefore $\phi_2(x_1,0,t)=0$ for any $x_1\in \R$ and $t\geq 0$. Let $a(t)=\phi_1(0,0,t)$. Then, it satisfies
	\[
	\begin{cases}
	a'(t) = -\pa_2\Del^{-1}T_{\ga}\om(a(t),0,t)\\
	a(0) = 0.
	\end{cases}
	\] 
	
	Similar to \eqref{condition.decomp.f}-\eqref{condition.decomp.fg}, we can easily see that the supports of $\om_f$ and $\om_g$ are apart from each other for a short time. Indeed, on $[0,t_A]$, $t_A = \frac 1{\ln\ln\ln A}$, 
	\EQN{
	\supp(\om_f(\cdot,t))&\subset B\bke{\supp(f), \frac 18 R_0}\subset\left\{x_1\leq -\frac{15}8 R_0\right\},\\ 
		\supp(\om_g(\cdot,t))&\subset B\bke{\supp(g_A), \frac 18 R_0}\subset B\bke{0,\frac 14 R_0},
	}
	provided that $A$ is sufficiently large. It follows that $\na^{\perp}\De^{-1}T_\ga \om_f$ is smooth and has Sobolev norm bounds on $B(0, \frac 14 R_0)\times [0,t_A]$, where the bounds depend only on $f$ and $R_0$. Therefore, we can expand it at the point $(a(t),0)$, which is in $B(0, \frac 18 R_0)$ for $0\leq t\leq t_A$, to get 
	\EQN{
		\na^{\perp}\Del^{-1}T_{\ga}\om_f(a(t)+y_1,y_2,t)
		&=\begin{pmatrix}
			a'(t)+\pa_{2}\Del^{-1}T_{\ga}\om_g(a(t),0,t)\\
			0
		\end{pmatrix}
		+b(t)
		\begin{pmatrix}
			-y_1 \\ y_2
		\end{pmatrix}
		+E(y,t)
	}
for any $(a(t)+y_1,y_2,t)\in B(0,\frac 14 R_0)\times [0,t_A]$. Here, $b(t) = \pa_{12}\Del^{-1}T_{\ga}\om_f(a(t),0,t)$ has a bound $|b(t)|\leq B_1$ for some $B_1=B_1(R_0,f)$ and a divergence-free vector $E$ can be chosen  
satisfying
\[
|E(y,t)|\leq B_2|y|^2, \quad
|DE(y,t)|\leq B_2|y|, \quad
|D^2E(y,t)|\leq B_2, \quad\forall y\in \R^2,
\]
for some $B_2=B_2(R_0,f)$. In the expansion, we use the oddness of $\om_f$ in $x_2$ and 
\[
\pa_1\De^{-1}T_{\ga}\om_f (a(t),0,t) 
= \pa_{11}\De^{-1}T_{\ga}\om_f (a(t),0,t) 
= \pa_{22}\De^{-1}T_{\ga}\om_f (a(t),0,t) =0.
\]	

We do the change of variables $(x_1,x_2,t) = (a(t)+y_1,y_2,t)$ and denote the solution in a new coordinate system $(y,t)$ by $\Om(y,t)= \om_g(a(t)+y_1,y_2,t) = \om_g(x_1,x_2,t)$. Then, the equation for $\Om$ on $\R^2\times[0,t_A]$ can be written as
\[
\begin{cases}
\pa_t \Om+\left(\na^{\perp}\Del^{-1}T_{\ga}\Om + B + E -C\right) \cdot\na \Om =0\\
\Om|_{t=0} = g_A,
\end{cases}
\] 
where $B$ and $C$ are
\[
B=b\begin{pmatrix}
-y_1\\ y_2
\end{pmatrix},
\quad
C=\begin{pmatrix}
-\pa_2\Del^{-1}T_{\ga}\Om(0,0,t) \\ 0
\end{pmatrix}.
\]
Also, we let $\Phi$ be a characteristic in a new coordinate, which solves
\[
\begin{cases}
\pa_t \Phi(y,t) = \left(\na^{\perp}\Del^{-1}T_{\ga} \Om 
+B
+E
-C\right)(\Phi(y,t),t)
\\
\Phi(y,0) = y. 
\end{cases}
\]

\medskip

From now on, without mentioning, we only consider $t\in [0,t_A]$.

We can easily check that $\phi^{-1}(x,t) = \Phi^{-1}(y,t)$ and $\Phi^{-1}(0,t) = \phi^{-1}(a(t),0,t) = 0$. Furthermore, $\Phi^{-1}$ satisfies
\[
\Phi_1^{-1}(y_1,y_2,t)=\Phi_1^{-1}(y_1,-y_2,t), \quad
\Phi_2^{-1}(y_1,y_2,t)=-\Phi_2^{-1}(y_1,-y_2,t).
\]
By Lemma \ref{supp.sol.lemma}, on the other hand, we have 
\[
|\Phi(y,t)|\le 4\cdot 2^{-A}, \quad\forall |y|\leq 2\cdot 2^{-A},
\]
for sufficiently large $A$. Also, if $|y|\le 4\cdot 2^{-A}$ and $t_A$ is sufficiently small, by finite speed propagation, $|\phi^{-1}(a(t)+y_1,y_2,t)|\leq 10\cdot 2^{-A}$. It follows that 
\EQ{\label{bdd.deform.mt}
\max_{0\leq t\leq t_A}\norm{D\Phi(\cdot,t)}_{L^\infty(|y|\le 2\cdot 2^{-A})}&\le 
\max_{0\leq t\leq t_A}\norm{D(\Phi^{-1})(\cdot,t)}_{L^\infty(|y|\le 4\cdot 2^{-A})}\\
&\le \max_{0\leq t\leq t_A}\norm{D\phi(\cdot,t)}_{L^\infty(|x|\le 10\cdot 2^{-A})}\\
&\le \ln^{\frac 14}\ln\ln\ln A = M_A.
}
Indeed, $(D\Phi(x,t))^{-1}=D(\Phi^{-1})(\Phi(x,t))$ and $(D\phi(x,t))^{-1}=D(\phi^{-1})(\phi(x,t))$ are used in the first and second inequalities, respectively. 

Then, by Lemma \ref{small.reisz.lem} with $\phi=\Phi^{-1}$, we have
\EQ{\label{small.Riesz1}
\sup_{0\leq t \leq t_A}\norm{\cR_{11}T_\ga \Om(\cdot,t)}_\infty +\norm{\cR_{22}T_\ga \Om(\cdot,t)}_\infty \leq \frac {C_\ga M_A}{\sqrt{\ln\ln A}} 
}
for $\cR_{ij}\om = \De^{-1}\pa_{ij}\om$ and for some constant $C_\ga>0$ depending only on $\ga$.

Now, we find a lower bound of $D\Phi$ which makes a contradiction to \eqref{bdd.deform.mt}.
From the equation for $\Phi$, we get
\EQ{\label{eqn.Dphi.12}
\begin{cases}
\pa_t D\Phi(y,t) = \left(D\na^{\perp}\Del^{-1}T_{\ga} \Om 
+b\begin{pmatrix}
-1 & 0\\  0 & 1
\end{pmatrix}
+DE\right)(\Phi(y,t),t)D\Phi(y,t)
\\
D\Phi(y,0) = I, 
\end{cases}
} 
and the derivative of the velocity can be rewritten as
\EQN{
D\na^{\perp}\Del^{-1}&T_{\ga} \Om 
+b\begin{pmatrix}
-1 & 0\\  0 & 1
\end{pmatrix}
+DE\\
&=\begin{pmatrix}
-\cR_{12}T_{\ga} \Om-b & 0\\
0			& \cR_{12}T_{\ga} \Om + b
\end{pmatrix} 
+ \begin{pmatrix}
0 & -\cR_{22}T_{\ga} \Om \\
\cR_{11}T_{\ga} \Om & 0
\end{pmatrix}
+ DE\\
&=\begin{pmatrix}
	-\cR_{12}T_{\ga} \Om -b & 0\\
	0			& \cR_{12}T_{\ga} \Om + b
\end{pmatrix} 
+P.
}
By Gr\"{o}nwall's inequality, we have
\EQ{\label{dphi.gr}
D\Phi&(y,t) 
= \exp \begin{pmatrix}
\int_0^t\la (y,s)-b(s)ds & 0\\
0			&\int_0^t -\la (y,s) + b(s) ds
\end{pmatrix}\\ 
&+
\int_0^t \exp \begin{pmatrix}
\int_\tau^t\la (y,s)-b(s)ds & 0\\
0			&\int_\tau^t -\la (y,s) + b(s) ds 
\end{pmatrix}
P(\Phi(y,\tau),\tau) D\Phi(y,\tau) d\tau
}
where $\la(y,t) =-\cR_{12}T_{\ga} \Om(\Phi(y,t),t)$. 

Since $|\Phi(y,t)|\le 4\cdot 2^{-A}$ for $|y|\le 2\cdot 2^{-A}$, $A\gg 1$, we have $|DE(\Phi(y,t),t)|\lec B_2 2^{-A}$. Combining \eqref{dphi.gr} with \eqref{bdd.deform.mt} and \eqref{small.Riesz1}, we obtain for $|y|\leq 2\cdot 2^{-A}$,
\[
\exp\left|\int_0^t\la (y,s)-b(s)ds\right|
\leq M_A + \frac{C_\gamma M_A^2}{\sqrt{\ln\ln A}}\max_{0\leq \tau\leq t}\exp\left(2\left|\int_0^\tau\la (y,s)-b(s)ds\right|\right).
\]
Then, by the continuation argument, we get
\[
\exp\left|\int_0^t\la (y,s)-b(s)ds\right| \leq 2M_A
\]
for sufficiently large $A$, so that we can consider the second term in \eqref{dphi.gr} as an error term. 

\medskip

The remaining analysis is similar to the proof of Proposition \ref{large.lagrangian}. 
Using $\Phi(0,t)=0$, it follows that for $|y|\le  2\cdot 2^{-A}$ 
\EQ{\label{approx.Phi.12}
\Phi(&y,t) 
= \Phi(y,t) - \Phi(0,t) = \int_0^1 \frac {\pa}{\pa \th} [\Phi(\th y,t)] d\th
= \left(\int_0^1  D\Phi(\th y,t) d\th \right) y\\
=& \left(y_1\int_0^1\exp \left( \int_0^t\la (\th y,s)-b(s)ds\right)d\th , 
y_2\int_0^1\exp \left( -\int_0^t\la (\th y,s)-b(s)ds\right)d\th \right) + e
}
where 
\[
|e(y,t)|\lesssim_{\ga} \frac{M_A^4}{\sqrt{\ln\ln A}}|y|.
\]
Since $\frac 1{M_A} \gg \frac{M_A^4}{\sqrt{\ln\ln A}}$ if $A\gg 1$ and $y_1\sim y_2$ for $y=(y_1,y_2)\in \supp(g_A)$, it follows that for sufficiently large $A$, $\Phi$ has a sign preserving property; 
\[
\Phi_1(y,t)>0, \quad \Phi_2(y,t)>0, \qquad y\in \supp(g_A)\cap 
\{ y_1>0, y_2>0\}
\]
\[
\Phi_1(y,t)<0, \quad \Phi_2(y,t)>0, \qquad y\in \supp(g_A)\cap 
\{ y_1<0, y_2>0\}
\]
Based on this, we get
\EQN{
\la(0,t)  &= -\cR_{12}T_\ga \Om(\Phi(0,t),t) = -\cR_{12}T_\ga \Om(0,t)\\
&= \int_{\R^2} K(-z,t)\Om(z,t) dz
= 2\int_{z_2>0} K(z,t)\Om(z,t) dz
=2\int_{z_2>0} K(\Phi(z,t),t)g_A(z) dz\\
&\geq \int_{z_1>0,z_2>0} K(\Phi(z,t),t)g_A(z) dz
}
where $K$ is the kernel of the operator $-\pa_{12}\De^{-1}T_{\ga}$. The fourth equality follows from the parity of $K$ and $\Om$ in $z_2$. 
The last inequality follows from the positiveness of the integrand on $\{z_1<0, z_2>0\}$.  

Note that if $z\in \supp(g_A)\cap \{z_1>0,z_2>0\}$, we have
\[
\frac 12 <\frac{z_1}{z_2}< 2
\]
and hence by \eqref{approx.Phi.12} 
\[
\frac 1{10M_A^2} <\frac {\Phi_1(z,t)}{\Phi_2(z,t)} <10 M_A^2. 
\] 
Also, we have $\frac{|z|}{M_A}\leq |\Phi(z,t)|\leq M_A |z|$ for $z\in \supp(g_A)$. 
Then, by Lemma \ref{estimate.K} and Lemma \ref{estimate.tdK}, we get
\EQN{
\int_{z_1>0,z_2>0} &K(\Phi(z,t),t)g_A(z) dz\\
&\gtrsim_\ga \int_{z_1>0,z_2>0}\frac{\Phi_1(z,t)\Phi_2(z,t)}{|\Phi(z,t)|^4} \ln^{-\ga}\left(e + \frac 1{|\Phi(z,t)|}\right) e^{-|\Phi(z,t)|^2} g_A(z)dz \\
&\geq \frac 1{M_A^2}\int_{z_1>0,z_2>0}
\frac{1}{\frac{\Phi_1(z,t)}{\Phi_2(z,t)}+\frac{\Phi_2(z,t)}{\Phi_1(z,t)}}\cdot \frac{1}{|z|^2} \ln^{-\ga}\left(e + \frac {M_A}{|z|}\right) e^{-M_A^2|z|^2} g_A(z)dz \\
&\gtrsim \frac {e^{\frac 34M_A^4}}{M_A^4(1+\ln(1+M_A))^{\ga}}e^{-M_A^4} \int_{z_1>0,z_2>0}\frac{1}{|z|^2} \ln^{-\ga}\left(e + \frac 1{|z|}\right) e^{-|z|^4} g_A(z)dz\\
&\gtrsim  e^{-M_A^4} \int_{z_1>0,z_2>0}\frac{1}{|z|^2} \ln^{-\ga}\left(e + \frac 1{|z|}\right) e^{-|z|^4} g_A(z)dz \\
&\gtrsim_{\ga} \frac {\sqrt{\ln\ln A}}{\ln\ln\ln A} e^{-M_A^4}
= \frac {\sqrt{\ln\ln A}}{(\ln\ln\ln A)^2},
}
provided that $A\gg 1$. Therefore, we get 
\EQN{
\max_{0\leq t\leq t_A}\norm{D\Phi(\cdot, t)}_{L^\infty(|y|\le 2\cdot 2^{-A})} 
&\geq \max_{0\leq t\leq t_A}|D\Phi(0,t)| \\
&\geq \exp\left(\left(\frac {C_{\ga}\sqrt{\ln\ln A}}{(\ln\ln\ln A)^2}-B_1\right)\frac 1{\ln\ln\ln A}\right) - 1,
}
which makes a contradiction to \eqref{bdd.deform.mt}
\[
\max_{0\leq t\leq t_A}\norm{D\Phi(\cdot, t)}_{L^\infty(|y|\le 2\cdot 2^{-A})}  
\leq \ln^{\frac 14}\ln\ln\ln A 
\]
for sufficiently large $A$. 
\end{proof}

Now, we give a proof of the main proposition. 

\noindent\textit{Proof of Proposition \ref{prop.cpt}.}  

\noindent\texttt{Step 1.}  Critical norm inflation of a local solution. 

By Lemma \ref{lld.cpt}, we can create a large Lagrangian deformation \eqref{lld.cpt.ineq.12} at the presence of $f$ satisfying \eqref{prop.f}. 
Then, similar to Proposition \ref{H1.norm.inflation.prop}, we can find a perturbed initial data $\td g_A\in C_c^{\infty}(|x|\lec 2^{-A})$ from $g_A$ such that it satisfies 
\EQN{
		\td g_A(x_1,x_2) &= -\td g_A(x_1,-x_2),\\
		\norm{\td g_A}_{\dot{H}^1}+ \norm{\td g_A}_\infty + \norm{\td g_A}_1 &+ \norm{\td g_A}_{\dot{H}^{-1}} \lec \ln^{-\frac 18}\ln\ln\ln A, 
}
and the smooth solution $\td\om^{(A)}$ to
\[
\begin{cases}
\pa\td\om^{(A)} + \na^{\perp}\Del^{-1}T_{\ga}\td\om^{(A)} \cdot \na \td\om^{(A)} =0\\
\td\om^{(A)}|_{t=0} = f+ \td g_A
\end{cases}
\]
has a decomposition $\td\om^{(A)} = \td\om^{(A)}_f+ \td\om^{(A)}_{\td g}$ where $\td\om^{(A)}_g$ satisfies 
\[
\max_{0\leq t\leq \frac 1{\ln\ln\ln A}}\norm{\na \td\om^{(A)}_g(\cdot,t)}_{2}
\geq \ln^{\frac 1{12}}\ln\ln\ln A,
\]
for sufficiently large $A$. Indeed, $\td\om^{(A)}_g$ solves
\[
\begin{cases}
\pa\td\om^{(A)}_g + \na^{\perp}\Del^{-1}T_{\ga}\td\om^{(A)} \cdot \na \td\om^{(A)}_g =0\\
\td\om^{(A)}_g|_{t=0} = \td g_A.
\end{cases}
\]

Then, we construct $g$ by choosing $A_0=A_0(\ep_0)\gg 1$ such that $g=\td g_{A_0}\in C_c^\infty(B(0,\ep_0))$ and $\td\om_g=\td\om^{(A_0)}_g$ satisfies 
\[
\norm{g}_{\dot{H}^1} + \norm{g}_\infty + \norm{g}_1+ \norm{g}_{\dot{H}^{-1}} <\ep_0,
\]
and
\[
\max_{0\leq t\leq \frac 1{\ln\ln\ln A_0}}\norm{\na \td\om_{g}(\cdot,t)}_{2}
>\frac 2{\ep_0}. 
\]
In particular, we can find $0<t_0\leq \frac 1{\ln\ln\ln A_0}< \ep_0$ such that
\EQ{\label{norm.inflation.Omg}
\norm{\na \td\om_{g}(\cdot,t_0)}_{2}
>\frac 2{\ep_0}. 
}

\bigskip

\noindent\texttt{Step 2.} Patch a function $h$.

Suppose that $h$ satisfies \eqref{condi.h} and $\de<1$. Let $\om$ be a solution to
\[
\begin{cases}
\pa_t \om +\na^{\perp}\Del^{-1}T_{\ga}\om\cdot \na\om = 0\\
\om|_{t=0} = f + g+ h. 
\end{cases}
\]
We decompose $\om = \om_f+\om_g +\om_h$ where $\om_f$ and $\om_g$ are defined as solutions to
\[
\begin{cases}
\pa_t \om_f +\na^{\perp}\Del^{-1}T_{\ga}\om\cdot \na\om_f = 0\\
\om|_{t=0} = f 
\end{cases}
\]
and
\[
\begin{cases}
\pa_t \om_g+\na^{\perp}\Del^{-1}T_{\ga}\om\cdot \na\om_g = 0\\
\om|_{t=0} = g, 
\end{cases}
\]
respectively. Since  
\[
\norm{\na^{\perp}\De^{-1}T_\ga \om}_\infty
\lec \norm{\om}_1 + \norm{\om}_\infty
= \norm{\om|_{t=0}}_1 + \norm{\om|_{t=0}}_\infty
\lec 1 + \norm{f}_1 + \norm{f}_\infty,
\]
similar to \eqref{condition.decomp.f} and \eqref{condition.decomp.g}, we can easily check $\om_f$, $\om_g$ and $\om_h$ satisfies \eqref{supp.omfgh}, provided that $t_0$ is sufficiently small. If necessary, we can adjust the choice of $A_0$ to make $t_0$ small enough.  

Now, recall \eqref{est.ug}. By the assumption \eqref{condi.h} on $h$ and \eqref{supp.omfgh}, we have
\[
\norm{\na^{\perp}\Del^{-1}T_{\ga}\om_h(\cdot,t)}_{L^{\infty}(B(\supp(f),\frac 18R_0)\cup B(0,\ep_0+\frac 18R_0))} \lesssim_{R_0} \norm{h}_1 \leq \del
\]
for any $0\leq t\leq t_0$.
Then, by the same arguments in Lemma \ref{simple.patching.2}, we get 
\[
\norm{(\om_g-\td\om_g)(\cdot,t_0)}_{H^2}
\leq \max_{0\leq t\leq t_0}\norm{((\om_f+\om_g)-\td\om)(\cdot,t)}_{H^2}
\leq C(\norm{f}_{H^3},R_0,\supp(f)) \de
\leq \frac 1{\ep_0}, 
\]
provided that $\del\in (0,1)$ is sufficiently small. 
Combining with \eqref{norm.inflation.Omg}, we obtain the desired inflation \eqref{norm.inflation.omg}. 

\hfill$\square$

Before we prove Theorem \ref{thm.cpt}, we need the following lemma for the uniqueness. 

\begin{lemma}\label{cpt.uniq}
Suppose that $f\in C_c^{\infty}(\R^2)$ with the compact support in $B(0,R)$ for some $R>0$ and $g\in L^{\infty}(\R^2)\cap \dot{H}^{-1}(\R^2)$ with $\norm{g}_{\infty}\leq M$ for some $M>0$. Let $\td \om$ be a smooth solution to
\[
\begin{cases}
\pa_t \td\om + \na^{\perp}\Del^{-1}T_{\ga}\td\om\cdot\na\td\om =0\\
\td\om|_{t=0} = f
\end{cases}
\] 
and $\om$ be a weak solution in $C([0,1];L^{1}(\R^2)\cap L^{\infty}(\R^2))$ to
\[
\begin{cases}
\pa_t \om + \na^{\perp}\Del^{-1}T_{\ga}\om\cdot\na\om =0\\
\om|_{t=0} = f+g
\end{cases}
\] 
satisfying $L^\infty$-norm preservation
\[
\norm{\om(\cdot,t)}_{\infty} = \norm{f+g}_{\infty}, \quad\forall 0\leq t\leq 1.
\]
Then, for any $\ep>0$, we can find a constant $\del=\del(\ep,f, M)>0$ such that if $\norm{g}_{\dot{H}^{-1}}<\del$,  
\[
\max_{0\leq t\leq 1}\norm{(\om-\td\om)(\cdot,t)}_{\dot{H}^{-1}(\R^2)}<\ep.
\]
Furthermore, under the additional assumption $g\in C_c^{\infty}(B(0,R))$, we have $\td\del=\td\del(\ep,R,f)>0$ such that if $\norm{g}_{\infty}<\td\del$,  
\[
\max_{0\leq t\leq 1}\norm{(\om-\td\om)(\cdot,t)}_{L^{\infty}(\R^2)}<\ep.
\]

\end{lemma}

\begin{proof}
The equation for $\eta = \om-\td\om$ is 
\[
\begin{cases}
\pa_t\eta + \na^{\perp}\Del^{-1}T_{\ga}\eta \cdot \na \om + \na^{\perp}\Del^{-1}T_{\ga}\td\om \cdot \na \eta =0\\
\eta|_{t=0} = g. 
\end{cases}
\] 	

Taking $\int_{\R^2}\cdot\La^{-2}\eta dx$, $\La = (-\Del)^{\frac 12}$, on both side of the equation, we have
\EQN{
\frac 12 \ddt \norm{\eta}_{\dot{H}^{-1}(\R^2)}^2
&\leq \left|\int \om (\na^{\perp}\Del^{-1}T_{\ga}\eta \cdot \na) \La^{-2}\eta dx\right|
+\left|\int \eta (\na^{\perp}\Del^{-1}T_{\ga}\td\om \cdot \na)\La^{-2}\eta dx\right|\\
&\leq \norm{\om}_{\infty}\norm{\eta}_{\dot{H}^{-1}(\R^2)}^2
+\norm{\La^{-1}\eta}_2\norm{[\La,\na^{\perp}\Del^{-1}T_{\ga}\td\om\cdot \na] \La^{-2}\eta}_2\\
&\lesssim (\norm{f}_{\infty}+\norm{g}_{\infty}
+\norm{D\na^\perp\De^{-1}T_{\ga}\td\om}_{\infty})\norm{\eta}_{\dot{H}^{-1}}^2.
}
 
Here, the second inequality follows from 
\[
\int \La^{-1}\eta (\na^{\perp}\Del^{-1}T_{\ga}\td\om\cdot \na)\La^{-1}\eta dx 
=
\frac 12\int (\na^{\perp}\Del^{-1}T_{\ga}\td\om\cdot \na)|\La^{-1}\eta|^2 dx 
=0
\]
and the third one from the commutator estimate
\[
\norm{\La(lm)-l(\La m)}_2 
\lesssim \norm{Dl}_{\infty}\norm{m}_2.
\] 
By Gr\"{o}nwall inequality, we have 
\[
\max_{0\leq t\leq 1}\norm{\eta(\cdot,t)}_{\dot{H}^{-1}(\R^2)}
\leq \norm{g}_{\dot{H}^{-1}(\R^2)}\exp(C(\norm{f}_{\infty}
+\norm{f}_{W^{1,4}(\R^2)}+M)).
\]
Therefore, for given $\ep>0$, we can find the desired $\del=\del(\ep, f,M)$.

Now, we further assume that $g$ is in $C_c^{\infty}(B(0,R))$. Then, the weak solution $\om$ becomes a smooth solution. The equation for $\eta$ can be rewritten as
\[
\pa_t \eta + \na^{\perp}\Del^{-1}T_{\ga}\om \cdot \na \eta +\na^{\perp}\Del^{-1}T_{\ga}\eta \cdot \na \td\om =0,
\] 
so that we have
\EQ{\label{est.eta.infty}
\norm{\eta(\cdot,t)}_{\infty}\leq \norm{g}_{\infty} + \int_0^t \norm{(\na^{\perp}\Del^{-1}T_{\ga}\eta)(\cdot,s)}_{\infty}\norm{\na \td\om(\cdot,s)}_{\infty} ds.
}

By the usual energy estimate, we have
\[
\max_{0\leq t\leq 1}\norm{\na\td\om(\cdot,t)}_{\infty}\lesssim_f 1.
\]
Using $f,g\in C_c^\infty(B(0,R))$ and Lebesgue measure preservation  of the supports $\om$ and $\td \om$,
\EQN{
\norm{(\na^{\perp}\Del^{-1}T_{\ga}\eta)(\cdot,s)}_{\infty}
&\lesssim 
\norm{\eta(\cdot,t)}_1^{\frac 12}\norm{\eta(\cdot,t)}_{\infty}^{\frac 12}\\
&\leq 
(|\supp(\om(\cdot,0))| + |\supp(\td \om(\cdot,0))|)^\frac12 \norm{\eta(\cdot,t)}_{\infty}
\lesssim R \norm{\eta(\cdot,t)}_{\infty}. 
}

Then, combining with \eqref{est.eta.infty} and using Gr\"{o}nwall inequality, we have
\[
\max_{0\leq t\leq 1}\norm{\eta(\cdot,t)}_{\infty}
\leq C(R,f)\norm{g}_{\infty}.
\]

This completes the proof.
\end{proof}

\bigskip
Finally, we find the compactly supported perturbation in our main theorem. 

\

\noindent\textit{Proof of Theorem \ref{thm.cpt}.}
Fix $0<\ep<\frac 1{200}$. Without loss of generality, we may assume the support of the given initial data lies on $\{x=(x_1,x_2): x_1\leq -1\}\cap B(0,\bar{R})$ for some $\bar{R}\geq 10$. (Otherwise, using translation invariant property of the solution, we apply the proof for a suitably translated initial data in $x_1$ direction. Note that the translated one is still odd in $x_2$.) Let $\{x_n=(x_n^1,0)\}$ be a sequence of centres with
\[
x_1^1 = 0, \qquad x_n^1 = \sum_{j=1}^{n-1} \frac 1{2^j} \quad\text{for } n\geq 2.
\]

Now, we construct sequences $\{\zeta_n\}_{n\in \N}\subset C_c^{\infty}(B(0,2^{-(n+1)}))$, $\{(\del_n,\td\del_n, t_n)\}_{n\in \N}\subset \R_+^3$ such that for any $n\in \N$, 
\begin{itemize}
	\item $\zeta_n$ is odd in $x_2$ and satisfies
	\EQN{
	\norm{\zeta_n} \equiv 
	\norm{\zeta_n}_{\dot{H}^1}+\norm{\zeta_n}_{\infty}&+\norm{\zeta_n}_1+\norm{\zeta_n}_{\dot{H}^{-1}} <\min\left(\frac {\ep}{2^n},\frac{\del_{n-1}}{2^{n-1}},\frac{\td\del_{n-1}}{2^{n-1}}\right),
	}
	where $\del_0=\td\del_0 = 1$. 
	\item for any $h\in C_c^{\infty}(\R^2)$ with 
	\EQ{\label{condition.h}
	\supp(h)&\subset \{x=(x_1,x_2):x_1\geq \frac 1{2^{n+1}}\}\\
	&\norm{h}_1 +\norm{h}_{\infty}\leq \del_n,
	}
	a smooth solution $\om$ to \eqref{main.eq} for the initial data 
	\[
	\om|_{t=0}(x) = a(x+x_n)+\sum_{j=1}^{n-1}\zeta_j(x-x_j+x_n) + \zeta_n(x) + h(x)
	\]
	has a decomposition
	\[
	\om = \om_{\leq n-1} + \om_n + \om_h
	\]
	such that the supports of $\om_{\leq n-1}$, $\om_n$, and $\om_h$ are disjoint for $t\in [0,t_n]$, and 
	\[
	\norm{\om_n(\cdot,t_n)}_{\dot{H}^1}>2^n. 
	\]
	\item $\{\del_n\}$ and $\{\td\del_n\}$ are decreasing sequences. Also, $t_n$ converges to $0$.  
	
	\item for any $g$ satisfying $\norm{g}_{\infty}\leq 1$ and $\norm{g}_{\dot{H}^{-1}}\leq \td\del_n$, 
	\EQ{\label{neg.sob.uniq}
	\max_{0\leq t\leq 1}\norm{(\td\om-\td\om_{\leq n})(\cdot,t)}_{\dot{H}^{-1}}<\frac 1{2^n}.
}
	where $\td\om\in C([0,1];L^1(\R^2)\cap L^\infty(\R^2))$ is a weak solution having $L^{\infty}$-norm preservation and $\td\om_{\leq n}$ is a smooth solution to \eqref{main.eq} for initial data
	\EQ{\label{ini.td.w}
	\td\om|_{t=0}(x) = a(x) + \sum_{j=1}^{n}\zeta_j(x-x_j)  + g,
	\quad \td\om_{\leq n} |_{t=0} = a(x) + \sum_{j=1}^{n}\zeta_j(x-x_j).
	}
	Furthermore, if $g \in C_c^{\infty}(B(0,\bar{R}))$ with $\norm{g}_{\infty}\leq \td\del_n$, we have
	\EQN{
	\max_{0\leq t\leq 1}\norm{(\td\om-\td\om_{\leq n})(\cdot,t)}_{\infty}<\frac 1{2^n}.
}
	 
\end{itemize}

The construction is based on induction. First, we choose $\zeta_1$, and $(\del_1,\td\del_1,t_1)$. By Proposition \ref{prop.cpt} with
\[
f = a(x)=a(x+x_1), \quad R_0 = \frac 14, \quad \ep_0 = \frac{\ep}{2},
\]
there exist an smooth function $\zeta_1$ odd in $x_2$ and compactly supported in $B(0,\frac 14)$, and positive constants $0<\del_1<\del_0$ and $0<t_1<\frac 12$ which satisfy the following:
\begin{itemize}
	\item \[
	\norm{\zeta_1}<\frac{\ep}{2}
	\]
	\item If $h\in C_c^{\infty}(\R^2)$ satisfies
	\EQN{
	\supp(h)&\subset\{x=(x_1,x_2)\in \R^2: x_1 \geq \frac 14 \},\\
	&\norm{h}_1+\norm{h}_{\infty}\leq \del_1,
	}		
	a smooth solution $\om$ to \eqref{main.eq} for the initial data
	\[
	\om|_{t=0} = a(x)+\zeta_1(x)+ h(x)
	\]
	has a decomposition 
	\[
	\om = \om_{a} + \om_1 + \om_h, \quad\text{ on }\R^2\times[0,t_1]
	\]
	such that the supports of $\om_a$, $\om_1$, and $\om_h$ are disjoint for $t\in [0,t_1]$ and 
	\[
	\norm{\om_1(\cdot,t_1)}_{\dot{H}^1}>2.
	\]
\end{itemize} 
Then, we apply Lemma \ref{cpt.uniq} for $f=a+\zeta_1$, $R=\bar{R}$, $M=1$, and $\ep=\frac 12$, so that obtain $0<\td\del_1\leq \td\del_0$ such that if $\norm{g}_{\infty}\leq 1$, and $\norm{g}_{\dot{H}^{-1}}\leq \td\del_1$, then we have
\[
\max_{0\leq t\leq 1}\norm{(\td\om-\td\om_{\leq 1})(\cdot,t)}_{\dot{H}^{-1}(\R^2)}<\frac 12,
\]
where $\td\om$ and $\td\om_{\leq 1}$ are solutions to \eqref{main.eq} for the initial data
\[
\td\om|_{t=0} = a + \zeta_1 + g, \quad \td\om_{\leq 1}|_{t=0} = a+\zeta_1. 
\]
Furthermore, if $g\in C_c^{\infty}(B(0,\bar{R}))$ satisfies $\norm{g}_{\infty}\leq \td\del_1$, then we have
\[
\max_{0\leq t\leq 1}\norm{(\td\om-\td\om_{\leq 1})(\cdot,t)}_{\infty}<\frac 12,
\] 
Therefore, we obtain the desired $\zeta_1$ and $(\del_1,\td\del_1,t_1)$.

Assume that we have $\{\zeta_j\}_{j=1}^n$ and $\{(\de_j,\td\de_j,t_j)\}_{j=1}^n$ satisfying all conditions above. Then, applying Proposition \ref{prop.cpt} for 
\[
f=a(x+x_{n+1})+\sum_{j=1}^{n}\zeta_j(x-x_j+x_{n+1}), \quad
R_0=\frac 1{2^{n+2}}, \quad
\ep_0=\min\left(\frac {\ep}{2^{n+1}},\frac{\del_n}{2^n} ,\frac{\td\del_n}{2^n} \right),
\]
we can find $\zeta_{n+1}\in C_c^{\infty}(B(0,2^{-(n+2)}))$ odd in $x_2$, and $0<\del_{n+1}\leq \del_n$ and $0<t_{n+1}<\frac {1}{2^{n+1}}$ such that
\begin{itemize}
	\item \[
	\norm{\zeta_{n+1}} 
	<\min\left(\frac {\ep}{2^{n+1}},\frac{\del_n}{2^n} ,\frac{\td\del_n}{2^n} \right).
	\]
	\item for any $h\in C_c^{\infty}(\R^2)$ with 
	\EQN{
		\supp(h)&\subset \{x=(x_1,x_2):x_1\geq \frac 1{2^{n+2}}\}\\
		&\norm{h}_1 +\norm{h}_{\infty}\leq \del_{n+1},
	}
	the smooth solution $\om$ to \eqref{main.eq} for the initial data 
	\[
	\om|_{t=0}(x) = a(x+x_{n+1})+\sum_{j=1}^{n}\zeta_j(x-x_j+x_{n+1}) + \zeta_{n+1}(x) + h(x)
	\]
	has a decomposition
	\[
	\om = \om_{\leq n} + \om_{n+1} + \om_h, \quad \text{ on }\R^2\times[0,t_{n+1}] 
	\]
	such that the supports of $\om_{\leq n}$, $\om_{n+1}$, and $\om_h$ are disjoint for $t\in [0,t_{n+1}]$, and
	\[
	\norm{\om_{n+1}(\cdot,t_{n+1})}_{\dot{H}^1}> {2^{n+1}}. 
	\]
\end{itemize}

Once we obtain $\zeta_{n+1}$, applying Lemma \ref{cpt.uniq} for $f(x) = a(x) + \sum_{j=1}^{n+1} \zeta_j(x-x_j)$, $R=\bar{R}$, $M=1$, and $\ep={2^{-(n+1)}}$, we can find $0<\td\del_{n+1}\leq \td\del_n$ such that for any $g$ with $\norm{g}_{\infty}\leq 1$ and $\norm{g}_{\dot{H}^{-1}}\leq \td\del_{n+1}$, we have
\[
\max_{0\leq t\leq 1}\norm{(\td\om-\td\om_{\leq n+1})(\cdot,t)}_{\infty} <\frac 1{2^{n+1}},
\]
where $\td\om$ and $\td\om_{\leq n+1}$ solves \eqref{main.eq} for the initial data
\[
\td\om|_{t=0} = a+ \sum_{j=1}^{n+1} \zeta_j(\cdot-x_j)+g, \quad
\td\om_{\leq n+1}|_{t=0} = a+\sum_{j=1}^{n+1} \zeta_j(\cdot-x_j).
\]
If $g$ further satisfies $g\in C_c^{\infty}(B(0,\bar{R}))$ and $\norm{g}_{\infty}\leq \td\del_{n+1}$, we get
\[
\max_{0\leq t\leq 1}\norm{(\td\om-\td\om_{\leq n+1})(\cdot,t)}_{\infty} <\frac 1{2^{n+1}}.
\]
Therefore, by the induction argument, we obtain the desired sequences $\{\zeta_n\}$, $\{(\del_n,\td\del_n, t_n)\}$.

Now, we set the perturbation as
\[
\zeta(x) = \sum_{j=1}^{\infty} \zeta_j(x-x_j).
\]
Obviously, the perturbation satisfies
\[
\norm{\zeta}_{\dot{H}^1}+\norm{\zeta}_{\infty}+\norm{\zeta}_1+\norm{\zeta}_{\dot{H}^{-1}}
=\norm{\zeta}\leq \sum_{j=1}^{\infty}\norm{\zeta_j}< \ep.
\]

Since $\zeta_{n+1}(\cdot-x_{n+1})\in C_c^{\infty}(B(0,\bar{R}))$ and $\norm{\zeta_{n+1}}_{\infty}\leq \norm{\zeta_{n+1}}\leq \td\del_n$, we plug $g=\zeta_{n+1}(\cdot-x_{n+1})$ into \eqref{ini.td.w}
to get 
\EQ{\label{cauchy.cc}
\max_{0\leq t\leq 1}\norm{(\td\om_{\leq n+1}-\td\om_{\leq n})(\cdot,t)}_{\infty}<\frac 1{2^{n}}.
}
Indeed, for any $n\in \N$, $\zeta_{n}(\cdot-x_{n})\in C_c^{\infty}(B(0,\bar{R}))$, and by finite speed propagation we have $\td\om_{\le n} \in C([0,1]\times\overline{B(0,R_*)})$ for some finite number $R_*$. Then, \eqref{cauchy.cc} implies that $\{\td\om_{\leq n}\}$ is Cauchy in $C([0,1]\times \overline{B(0,R_*)})$, and hence we have its limit $\om\in  C([0,1];C_c(\R^2))$.  On the other hand, since $L^\infty$-norm of $\td\om_{\leq n}$ is preserved for any $n\in \N$, so is that of $\om$. 

Now, we check that $\om$ is the unique weak solution in $C([0,1];L^1(\R^2)\cap L^{\infty}(\R^2))$ to the equation \eqref{main.eq} for the initial data
\EQ{\label{ic}
\om|_{t=0} = a + \zeta,
}
having $L^\infty$-norm preservation. Since $\td\om_{\leq n}$ is smooth solution to \eqref{main.eq}, it satisfies for any $\ph\in C^1([0,1];C_c^1(\R^2))$ and $n\in \N$, 
\[
\int_{\R^2} \td\om_{\leq n}(x,1)\ph(x,1) dx = 
\int_{\R^2} \td\om_{\leq n}(x,0)\ph(x,0) dx
+\int_0^1\int_{\R^2}(\pa_s\ph+\na^{\perp}\Del^{-1}T_{\ga}\td\om_{\leq n}\cdot\na\ph)\td\om_{\leq n} dxds.
\]
Sending $n$ to infinity, $\om$ solves \eqref{main.eq} in a weak sense.  Then the uniqueness follows from \eqref{neg.sob.uniq}. Indeed, for any weak solution $\bar{\om}\in C^1([0,1];L^1(\R^2)\cap L^\infty(\R^2))$ to \eqref{main.eq} for the same initial data with $\om$ having $L^\infty$-norm preservation, we have
\[
\max_{0\leq t\leq 1}\norm{\bar{\om}-\td\om_{\leq n}(\cdot,t)}_{\dot{H}^{-1}(\R^2)}<\frac1{2^n},
\] 
for sufficiently large $n$. Here, we use $\sup_{j}\norm{\zeta_j}_{\infty}\leq 1$ and 
\[
\sum_{j=n+1}^{\infty}\norm{\zeta_j(\cdot-x_j)}_{\dot{H}^{-1}(\R^2)}
<\sum_{j=n+1}^{\infty}\frac{\td\del_{j-1}}{2^{j-1}}
\leq \td \del_n \sum_{j=n+1}^{\infty}\frac{1}{2^{j-1}}\leq \td\del_n.
\]
Therefore, if the weak solution is not unique, i.e., $\bar\om\neq \om$, then it makes a contradiction to  
\[
\max_{0\leq t\leq 1}\norm{(\bar{\om}-\om)(\cdot,t)}_{\dot{H}^{-1}(\R^2)}<\frac1{2^{n-1}}, \quad\forall n\in \N.
\] 
Therefore, we obtain the uniqueness.  

Finally, since $h=\sum_{j=n+1}^{\infty}\zeta_j(x-x_j+x_n)$ satisfies the conditions \eqref{condition.h}, we have
\EQ{\label{this}
\norm{\om(\cdot,t_n)}_{\dot{H}^1}\geq \norm{\om_n(\cdot,t_n)}_{\dot{H}^1}> 2^n.
}
Indeed, in Proposition \ref{prop.cpt}, the assumption $h\in C_c^{\infty}(\R^2)$ can be dropped if we have 
a unique weak solution $\om\in C([0,1];L^1(\R^2)\cap L^\infty(\R^2))$ to \eqref{main.eq} with initial data $\om|_{t=0}=f+g+h$. This leads to \eqref{this}.
 
Using the continuity of $\norm{\om_n(\cdot,t)}_{H^1}$, we have short time interval $[t_n^l,t_n^r]$, $t_n^l\leq t_n\leq t_n^r$ such that $t_n^r$ converges to $0$ and 
\[
\norm{\om(\cdot,t)}_{\dot{H}^1}> n, \quad\forall t_n^l\leq t\leq t_n^r. 
\]
This implies the desired critical Sobolev norm inflation.

\hfill$\square$
\section{Appendix}
In this section, we provide proofs of some inequalities for self-containedness. 

\subsection{Kernel for the velocity}
In this section, we estimate the kernel $H$ in the velocity $u= \na^\perp \De^{-1} T_\ga\om = H\ast \om$.

\begin{lemma}\label{ker.H.lem} Let $\ga>0$ and $H$ is the kernel of the multiplier $\na^\perp\Del^{-1}T_\ga$, where $T_\ga$ is either 
\[
T_\ga = \ln^{-\ga}(e-\Del), \quad\text{or}\quad
T_\ga = \ln^{-\ga}(e+|\na|).
\] 
Then, for each $\al$ with $|\al|\geq 0$, we have 
\EQ{\label{est.H}
|\pa^\al H(x)|\lesssim_{\al}\frac 1{|x|^{|\al|+1}}, \qquad\forall x\neq 0.
}
\end{lemma}

\begin{proof}
By a similar argument in Lemma 
\ref{estimate.K} and Lemma \ref{estimate.tdK}, we have an explicit expression of the kernel $H^2$ of the multiplier $\na^\perp \De^{-1}\ln^{-\ga}(e-\De)$,
\begin{align}\label{H2}
H^2(x)=
\frac{C}{\Ga(\ga)}\frac{x^\perp}{|x|^2}\int_0^{\infty}\frac 1{\Ga(t)}\int_0^{\infty}e^{-e\be}(1-e^{-\frac {|x|^2}{4\be}})\be^t \frac{d\be}{\be} t^{\ga}\frac{dt}t =: \frac{x^\perp}{|x|^2} H_r^2(x)  
\end{align}
where $x^\perp = (-x_2,x_1)$ for some absolute constant $C>0$. 

Also, the kernel $H^1$ of the multiplier $\na^\perp \De^{-1}\ln^{-\ga}(e+|\na|)$ is
\EQ{\label{H1}
H^1(x)
&=
 \frac{\td C}{\Ga(\ga)}\frac{x^\perp}{|x|^2}\int_0^{\infty}\frac {1}{\Ga(t)} \int_0^{\infty}e^{-\tau}
	\int_0^{\infty} e^{-e\be}\left(1 -e^{-\frac{\tau|x|^2}{{\be^2}}}\right)  \be^t \frac{d\be}{\be}\tau^{-\frac 12}d\tau t^{\ga}\frac{dt}t,\\
&=: \frac{x^\perp}{|x|^2} H_r^1(x)
}
for some constant $\td C>0$.

Using $|t^ne^{-t}|\leq C(n)$ for any $t\geq 0$, we have for each $|\al|\geq 0$,
\[
|\pa^{\al}(1-e^{-\frac{|x|^2}{4\be}})|\lesssim_{\al} \frac {1}{|x|^{|\al|}}\qquad\forall x\neq 0, \ \be>0,
\]
where the constant in the inequality is independent of $\be$. 
Since
\[
\frac 1{\Ga(\ga)}\int_0^{\infty}\frac {1}{\Ga(t)}\int_0^{\infty}e^{-e\be}\be^{t} \frac{d \be}{\be} t^{\ga}\frac{dt}{t}= \frac 1{\Ga(\ga)}\int_0^{\infty}e^{-t} t^{\ga}\frac {dt}t = 1,
\]
we can easily get
\[
|\pa^\al H_r^2(x)| \lesssim_\al \frac 1{|x|^{|\al |}}.
\]

On the other hand, we have
\EQN{
	\frac 1{\Ga(\ga)}\int_0^{\infty}\frac {1}{\Ga(t)} \int_0^{\infty}e^{-\tau}
	\int_0^{\infty} e^{-e\be}  \be^t \frac{d\be}{\be}\tau^{-\frac 12}d\tau t^{\ga}\frac{dt}t 
	&=\int_0^{\infty}e^{-\tau}\tau^{-\frac 12}d\tau
	\lesssim 1.
	}
Therefore, we also obtain
\[
|\pa^\al H_r^1(x)| \lesssim_\al \frac 1{|x|^{|\al |}}.
\]

Finally, since for each $|\al|\geq 0$, we have
\[
\left|\pa^{\al}\left(\frac{x_i}{|x|^2}\right)\right|\lesssim_{\al} \frac 1{|x|^{|\al|+1}}, \qquad \forall x\neq 0,
\]	
the desired estimate \eqref{est.H} follows easily.
\end{proof}

\subsection{Operator norm of $T_\ga$ on $L^p$ }

In this section, we show that $T_\ga$ is bounded in $L^p$ with its operator norm $\norm{T_\ga}_{L^p\to L^p}=1$. 

\begin{lemma}\label{operator.T} Let $\ga>0$ and $f\in C_c^{\infty}(\R^2)$. For any $1\leq p\leq \infty$, we have
	\[
	\norm{T_{\ga}f}_p\leq\norm{f}_p.
	\]
\end{lemma}
\begin{proof} Let $K_{\ga}$ be the kernel for $T_{\ga}$. In other words, $T_{\ga}f = K_{\ga}\ast f$. Then, by Young's inequality, it is enough to show that $\norm{K_\ga}_1=1$. First, consider $T_{\ga} = \ln^{-\ga}(e-\Del)$. Since we have
\EQN{
 \ln^{-\ga}(e+|\xi|^2) 
=\frac 1{\Ga(\ga)}\int_0^{\infty}\frac 1{\Ga(t)}\int_0^\infty e^{-e\be} e^{-|\xi|^2 \be}\be^t \frac{d\be}{\be} t^{\ga} \frac{dt}t,	
}
we take the inverse Fourier transform to get the corresponding kernel 
\EQN{
K_{\ga}(x) 
=\frac 1{\Ga(\ga)}\int_0^{\infty}\frac 1{\Ga(t)}\int_0^\infty e^{-e\be} e^{\be\Del }\del_0(x)\be^t \frac{d\be}{\be} t^{\ga} \frac{dt}t. 
}
Therefore using $\norm{e^{\be\Del }\del_0}_1 =1$, we can easily get $\norm{K_{\ga}}_1=1$. Here, $e^{t\De}\del_0$ is the usual heat kernel. 

Similarly, when $T_{\ga} = \ln^{-\ga}(e+|\na|)$, the integral expression of the kernel is
\EQN{
K_{\ga}(x) 
=\frac 1{\Ga(\ga)}\int_0^{\infty}\frac 1{\Ga(t)}\int_0^\infty e^{-e\be} \frac 1{\sqrt{\pi}}\int_0^\infty e^{-\tau}e^{\frac{\be}{4\tau}\Del }\del_0(x)\tau^{-\frac 12}d\tau \be^t \frac{d\be}{\be} t^{\ga} \frac{dt}t,
}
and hence again $\norm{e^{\frac{\be}{4\tau}\Del }\del_0}_1=1$ implies $\norm{K_{\ga}}_1=1$.
\end{proof}

\section*{Acknowledgments}
The research of the author was partially supported by NSERC grant 261356-13 (Canada).


\end{document}